\documentclass[11pt]{article}
\usepackage{amsmath, amsthm, amssymb, graphics, graphicx}
\usepackage{hyperref}
\usepackage{cleveref}
\usepackage[none]{hyphenat}
\usepackage{tocloft}
\usepackage{cancel}
\usepackage{multicol}
\usepackage{tikz}

\DeclareMathOperator{\lastBorn}{\ell}
\DeclareMathOperator{\choosePath}{c}

\usepackage[all,cmtip]{xy}

\DeclareMathOperator{\Forbs}{Forb^*}

\newcommand{\sm}{\setminus}

\binoppenalty=\maxdimen
\relpenalty=\maxdimen
\usepackage{indentfirst}
\usepackage{enumitem}
\setlist[itemize]{noitemsep, topsep=0pt}
\usepackage{array}
\newcolumntype{L}[1]{>{\raggedright\let\newline\\\arraybackslash\hspace{0pt}}m{#1}}
\newcolumntype{C}[1]{>{\centering\let\newline\\\arraybackslash\hspace{0pt}}m{#1}}
\newcolumntype{R}[1]{>{\raggedleft\let\newline\\\arraybackslash\hspace{0pt}}m{#1}}

\newtheorem{theorem}{Theorem}[section]
\newtheorem{lemma}[theorem]{Lemma}
\newtheorem{corollary}[theorem]{Corollary}

\newtheorem{remark}[theorem]{Remark}

\title{Burling graphs revisited, part III:\\ Applications to $\chi$-boundedness}
\author{Pegah Pournajafi\thanks{Univ Lyon, EnsL, UCBL, CNRS, LIP,
		F-69342, LYON Cedex 07, France. Partially supported by the LABEX
		MILYON (ANR-10-LABX-0070) of Universit\'e de Lyon, within the
		program ‘‘Investissements d'Avenir’’ (ANR-11-IDEX-0007) operated
		by the French National Research Agency (ANR) and by Agence
		Nationale de la Recherche (France) under research grant ANR
		DIGRAPHS ANR-19-CE48-0013-01.}~~and Nicolas Trotignon\footnotemark[1]}

\begin{document}

\maketitle

	\begin{abstract}
		The Burling sequence is a sequence of triangle-free graphs of unbounded chromatic number. The class of Burling graphs consists of all the induced subgraphs of the graphs of this sequence. 
		
		In the first and second parts of this work, we introduced derived graphs, a class of graphs, equal to the class of Burling graphs, and proved several geometric and structural results about them. 
		
		In this third part, we use those results to find some Burling and non-Burling graphs, and we see some applications of this in the theory of $\chi$-boundedness. Specifically, we show that several graphs, like $ K_5 $, some series-parallel graphs that we call necklaces, and some other graphs are not weakly pervasive.
	\end{abstract}

	\section{Introduction}

	The chromatic number of a graph $ G $, denoted by $ \chi (G) $, is the smallest number $ k $ such that we can partition the vertex set of $ G $ into $ k $ stable sets. The clique number of $ G $, denoted by $ \omega (G) $, is the maximum number of mutually adjacent vertices in $ G $. For every graph $G$, we have  $\chi(G) \geq \omega(G)$, but the converse inequality is false in general. It was known even from the early days of graph theory that there are triangle-free graphs with arbitrarily large chromatic number. For instance, Tutte's construction (for the definition and more examples see Section 2 of~\cite{Scott18}). The study of the graphs satisfying several variants of the converse inequality was the	object of many researches, and led to defining perfect graphs (see~\cite{Trotignon13} for more details), and a generalization, $ \chi$-bounded classes (see~\cite{Gyarfas87}). A class $ \mathcal C $ of graphs is called \emph{$ \chi $-bounded} with a \emph{binding function} $ f $ if for every $ G \in \mathcal C $, and for every induced subgraph $ H $ of $ G $ we have $ \chi(H) \leq f(\omega(H)) $. A function $ f $ is a \emph{$\chi$-binding function} for a class $ \mathcal C $ of graphs if for every $ G \in \mathcal C $, and for every induced subgraph $ H $ of $ G $, $ \chi(H) \leq f(\omega(H)) $. A class $ \mathcal C $ is called \emph{$\chi$-bounded} if there is a $\chi$-binding function for it. 
	
	It is in particular interesting to study the $\chi$-boundedness of specific classes of graphs. For a graph $ H $, a graph not including $ H $ as an induced subgraph is called \emph{$H$-free}, and one can ask whether the class of all $ H $-free graphs is $\chi$-bounded for a given graph $ H $. Erd\H os showed in~\cite{Erdos59} that for any given integers $ g $ and $ k $, there are graphs with girth larger than $ g $ and chromatic number larger than $ k $. Therefore for any given $ H $ with cycles, we can find graphs of arbitrary large chromatic number and with girth larger than the size of the biggest cycle in $ H $. Hence if the class of all $H$-free graphs is a $ \chi $-bounded class, then $ H $ must be a forest. Gy{\'a}rf{\'a}s~\cite{Gyarfas87}, and Sumner~\cite{Sumner81} conjectured the inverse: they conjectured that the class of $ F $-free graphs is $\chi$-bounded for every forest $ F $. It is possible to reduce this conjecture to the case that $ F $ is a tree (see Section 3 of~\cite{Scott18}). 
	
	In~\cite{Scott97}, Scott proved a weakening of the Gy{\'a}rf{\'a}s-Sumner conjecture. A subdivision of a graph $ H $ is a graph obtained from $ H $ by replacing some of its edges by paths of length at least~1. (Hence, every graph is a subdivision of itself.) We denote by $ \Forbs(H) $ the class of all graphs which do not contain any subdivisions of $ H $ as an induced subgraph. Scott proved that for every tree $ T $, the class $ \Forbs(T) $ is $\chi$-bounded. (Theorem 1 of~\cite{Scott97}.) Notice that if $ H $ has cycles, then subdivisions of $ H $ can have arbitrary large cycles, and therefore the previous idea cannot be applied here. In fact, $ \Forbs(H) $ might be $\chi$-bounded even if $ H $ is not a tree. In~\cite{Scott97}, Scott conjectured (Conjecture~8) that $ \Forbs(H) $ is $\chi$-bounded for every graph $ H $. 
	
	Pawlik, Kozik, Krawczyk, Laso\'n, Micek, Trotter, and Walczak disproved Scott's conjecture in~\cite{Pawlik2012Jctb}. But even after that, it remained an interest to know for which graphs $ H$, $ \Forbs(H) $ is $\chi$-bounded. A graphs $ H $ is called \emph{weakly pervasive} if $ \Forbs(H) $  is $\chi$-bounded. There still exists not even a good conjecture for the characterization of weakly pervasive graphs, and it is of interest to find examples of graphs that are or are not weakly pervasive. 
	
	One method for finding graphs that are not weakly pervasive, which is the core of all the researches so far in finding these graphs, is using some appropriate classes of graphs: suppose that you have a class $ \mathcal C $ of graphs which is not $\chi$-bounded, and moreover does not contain any subdivision of some given graph $ H $. This shows that $ \mathcal C \subseteq \Forbs(H) $ and thus $ H $ is not a weakly pervasive graph. The class of Burling graphs, introduced by Burling~\cite{Burling65} in 1965, is not $\chi$-bounded, and moreover, there are many graphs which are not in the class. So, the class of Burling graphs is a tool for finding graphs that are not weakly pervasive. 
	
	In~\cite{Pawlik2012Jctb}, where Pawlik, Kozik, Krawczyk, Laso\'n, Micek, Trotter, and Walczak disproved Scott's conjecture, they found Burling graphs as a subclass of \emph{line segment graphs}, and then they showed that some graphs (e.g.\ 1-subdivision of any non-planar graph) are not line segment graphs. In~\cite{Chalopin2014} also, Chalopin, Esperet, Li, and Ossona
	de Mendez used the same technique, and proved that Burling graphs are a subclass of \emph{restricted frame graphs} (as already suggested in \cite{Pawlik2015}), and then found many graphs which are not restricted frame graphs and neither are their subdivisions. All these suggest that understanding Burling graphs is of great importance in the field of $\chi$-boundedness. 
	
	In~\cite{Part1}, the first part of this work, with the goal of understanding Burling graphs better, we defined many classes of graphs, all equal to Burling graphs. In particular, we defined the class of \emph{derived} graphs, and we studied their structure in details in~\cite{Part2}, the second part of this work. In the next section, we give a summary of some of the results of the previous parts that we are going to use in this article. Then, in the rest of the article, we use those results to find several examples of new graphs that are not weakly pervasive:	
	\begin{itemize}
		\item[-] $ K_5 $ and thus any $ K_n $ for $ n\geq 5 $, in Section~\ref{section:Kn},
		\item[-] Specific necklaces in Section~\ref{section:necklace},
		\item[-] Some examples of dumbbell graphs, which contain vertex cuts, in Section~\ref{section:dumbbells},
		\item[-] Some graphs of different kinds in Section~\ref{section:Misc}.
	\end{itemize}

	\section{Summary of results from Parts I and II} \label{section:summary}
	
	In this section, we summarize the previous results on Burling
        graphs, mainly from~\cite{Part1} and~\cite{Part2}, that we
        need for the next sections.  We do not recall the most
        classical notation.

	\subsection*{Derived graphs}

	A \emph{Burling tree} is a 4-tuple
	$(T,r, \lastBorn, \choosePath)$ in which:
	
	\begin{enumerate}[label=(\roman*)] 
		\item $ T $ is a rooted tree and $ r $ is its root,
		\item $\lastBorn$ is a function associating to each vertex $v$ of $T$
		which is not a leaf, one child of $v$ which is called the
		\emph{last-born} of $v$,
		\item $\choosePath$ is a function defined on the vertices of $ T $: if
		$ v $ is a non-last-born vertex of $ T $ other than the root, then
		$ \choosePath $ associates to $ v $ the vertex-set of a (possibly empty)
		branch in $T$ starting at the last-born of the parent of $v$. If $v$ is a
		last-born or the root of $T$, then we define
		$ \choosePath(v) = \varnothing $. We call $ \choosePath $ the \textit{choice
			function} of $ T $.
	\end{enumerate}
	
	By abuse of notation, we often use $T$ to denote the 4-tuple.

	The oriented graph $G$ \emph{fully derived} from the Burling tree $T$
	is the oriented graph whose vertex-set is~$V(T)$ and~ $uv \in A(G) $ if and only if $v$ is a vertex in~$\choosePath(u)$.  A
	non-oriented graph $G$ is \emph{fully derived} from $T$ if it is the
	underlying graph of the oriented graph fully derived from $T$.
	
	A graph (resp.\ oriented graph) $G$ is \emph{derived} from a Burling
	tree~$T$ if it is an induced subgraph of a graph (resp.\ oriented
	graph) fully derived from~$T$. The oriented or non-oriented graph
	$G$ is called a \emph{derived graph} if there exists a Burling
	tree~$T$ such that~$G$ is derived from~$T$. See Figure~\ref{pic:K5sdB-minus-1vertex}.

	For more about the definition of derived graphs and their basic properties, see Section 3 of~\cite{Part1}.

	Because of the theorem below, we do not repeat the definition
        of Burling graphs here. For their definition see Definition
        4.1 of~\cite{Part1}. We just remind that the class of Burling
        graphs have unbounded chromatic number. It is first proved by
        Burling in~\cite{Burling65}, but the proof can be find in many
        references, for example, see~\cite{Pawlik2012Jctb}
        or~\cite{Part1}.
	
	\begin{theorem}[\cite{Part1}, Theorem
          4.9] \label{thm:BG=derived} The class of non-oriented
          derived graphs is the same as the class of Burling graphs.
	\end{theorem}
		
	By Theorem~\ref{thm:BG=derived}, we may interchangeably use
        the words derived graph and Burling graph.
	
	In this article, following the style of the first two parts of
        this work, for drawing the derived graph representation of
        graphs, we show the last-born of a vertex of the Burling tree
        as its right-most child. Also, we show the edges of the
        Burling tree with black and the arcs of the graph derived from
        it with red. Finally, we denote the vertices of the Burling
        tree which are in the graph derived from it with black, and
        the rest of its vertices (the \textit{shadow vertices}) with
        white.
	
	\subsection*{Subdivisions of $K_4$ and wheels}	
	
	The following lemma characterizes the subdivisions of $K_4$
        which are Burling graphs.
        
	\begin{lemma}[\cite{Part2}, Theorem
          7.3] \label{thm:subdivisions-of-K4} Let $G$ be a
          non-oriented graph obtained from $K_4$ by subdividing
          edges. Then $G$ is a Burling graph if and only if $G$
          contains four vertices $a$, $b$, $c$, and $d$ of degree 3
          such that $ab, ac \in E(G)$ and $ad, bc \notin E(G)$.
	\end{lemma}

	Let $ G $ be a subdivision of $K_4 $, which is a Burling
        graph. We say that $ G $ is of type $ i $, for $ i =2,3,4 $,
        if exactly $ i $ edges of the $ K_4 $ are properly subdivided
        to obtain $ G $. Notice that by
        Lemma~\ref{thm:subdivisions-of-K4}, up to symmetry, there is
        only one possibility for the choice of the edges to subdivide
        for each type.  See Figure~\ref{fig:sd-of-K4}.

	\begin{figure}
		\centering
		\begin{center}
			\begin{tikzpicture}[scale=.35]
			% the main vertices
			\filldraw[black] (-4,0) circle (5pt) node[anchor=north east] {$b$};
			\filldraw[black] (4,0) circle (5pt) node[anchor=north west] {$c$};
			\filldraw[black] (0,2) circle (5pt) node[anchor=south east] {$d$};
			\filldraw[black] (0,6) circle (5pt) node[anchor=south east] {$a$};
			% non-main vertices
			\filldraw[black] (0,4) circle (5pt);
			\filldraw[black] (0,0) circle (5pt);
			% edges
			\draw[black] (-4,0) -- (0,6) -- (4,0) -- (0,0);
			\draw[black] (-4,0) -- (0,2) -- (4,0);
			\draw[black] (0,2) -- (0,4);
			\draw[dashed] (0,4) -- (0,6);
			\draw[dashed] (0,0) -- (-4,0);
			\end{tikzpicture} 
			\hspace*{.6cm}
			\begin{tikzpicture}[scale=.35]
			% the main vertices
			\filldraw[black] (-4,0) circle (5pt) node[anchor=north east] {$b$};
			\filldraw[black] (4,0) circle (5pt) node[anchor=north west] {$c$};
			\filldraw[black] (0,2) circle (5pt) node[anchor=south east] {$d$};
			\filldraw[black] (0,6) circle (5pt) node[anchor=south east] {$a$};
			% non-main vertices
			\filldraw[black] (0,4) circle (5pt);
			\filldraw[black] (0,0) circle (5pt);
			\filldraw[black] (-2,1) circle (5pt);
			% edges
			\draw[black] (-4,0) -- (0,6) -- (4,0) -- (0,0);
			\draw[black] (0,2) -- (0,4);
			\draw[dashed] (0,0) -- (-4,0);
			\draw[dashed] (0,4) -- (0,6);
			\draw (0,2) -- (4,0);
			\draw (-4,0) -- (-2,1);
			\draw[dashed] (-2,1) -- (0,2);
			\end{tikzpicture}
			\hspace*{.6cm}
			\begin{tikzpicture}[scale=.35]
			% the main vertices
			\filldraw[black] (-4,0) circle (5pt) node[anchor=north east] {$b$};
			\filldraw[black] (4,0) circle (5pt) node[anchor=north west] {$c$};
			\filldraw[black] (0,2) circle (5pt) node[anchor=south east] {$d$};
			\filldraw[black] (0,6) circle (5pt) node[anchor=south east] {$a$};
			% non-main vertices
			\filldraw[black] (0,4) circle (5pt);
			\filldraw[black] (0,0) circle (5pt);
			\filldraw[black] (-2,1) circle (5pt);
			\filldraw[black] (2,1) circle (5pt);
			% edges
			\draw[black] (-4,0) -- (0,6) -- (4,0) -- (0,0);
			\draw (-4,0) -- (-2,1);
			\draw[dashed] (-2,1) -- (0,2);
			\draw[dashed] (0,4) -- (0,6);
			\draw[black] (0,2) -- (0,4);
			\draw[dashed] (-4,0) -- (0,0);
			\draw[dashed] (0,2) -- (2,1);
			\draw (2,1) -- (4,0);
			\end{tikzpicture}
		\end{center} 
		\vspace*{-.6cm}
		\caption{Subdivisions of $K_4 $, in order from left to right: type 2, type 3, and type 4. Only dashed edges can be subdivided. 
		%$ S_2 $, $ S_3 $, and $ S_4 $. Only the dashed edges can be subdivided.
		} \label{fig:sd-of-K4}
	\end{figure}
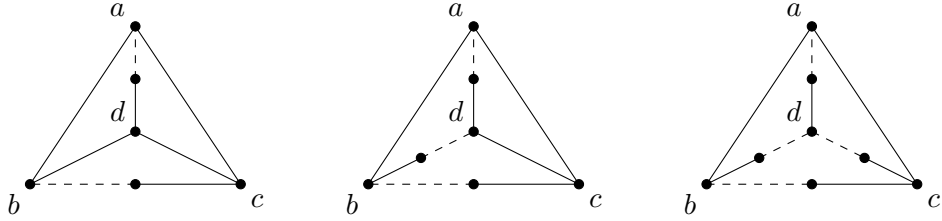
	
	A \emph{wheel} consists of a hole~$ C $ and a vertex~$ v $
        which has at least~3 neighbors in $ C $. 
        As explained in \cite{Part2} (Theorem 7.2), wheels are not Burling graphs. 
\subsection*{Star cutsets}
	
A \emph{full in-star cutset} in an oriented graph $G$ is a set
$S = N^{-} [v]$ for some vertex $v \in V (G)$ such that $G \sm S$ is
disconnected. In such case, we say that the start cutset is
\emph{centered} at $ v$.

An \emph{in-tree} is any oriented graph obtained from a rooted tree
$(T, r)$ by orienting every edge towards the root.  Formally, $e=uv$
is oriented from $u$ to $v$ if an only if $v$ is on the unique path of
$T$ from $u$ to $r$.  Notice that in an in-tree, every vertex but the
unique sink, has a unique out-neighbor. An \emph{in-forest} is an
oriented forest whose connected components are in-trees. A \emph{leaf}
in an in-tree is a source with exactly one out-neighbor (so, by
definition, the root is not a leaf, even if it has degree~1).
	
An \emph{oriented chandelier} is any oriented graph $G$ obtained
from an in-tree $G'$ with at least~2 leaves by adding a vertex $v$ and
all arcs $uv$ where $u$ is a leaf of $G'$.  Observe that $v$ is a sink
and all its neighbors are sources of degree~2.  We call here
\emph{non-oriented} chandelier the underlying graph of an oriented
chandelier.  Note that in~\cite{Chalopin2014}, \emph{chandelier} has a
slightly different meaning. 
		
\begin{theorem}[\cite{Part2}, Theorem
  5.5] \label{thm:starcutset-chandelier-degree1} If $G$ is an oriented
  Burling graph, then $G$ has a full in-star cutset, or $G$ is an
  oriented chandelier or $G$ contains a vertex of degree at most 1.
\end{theorem}

Observe that in any chandelier (oriented or not), there exist a vertex
($ v $ in the definition above) which is contained in every cycle of
the graph.  The following is observed in~\cite{Part2}, as a direct
consequence of the definition of \textit{k-sequential graphs}. See
Section~4 of~\cite{Part2}. In fact, chandeliers are 2-sequential
graphs, a subclass of Burling graphs.

\begin{lemma} \label{thm:chandeliers-are-Burling} Every (non-oriented)
  chandelier is a Burling graphs.
\end{lemma}

\subsection*{Holes}
	
A \emph{hole} in a graph G is a chordless cycle of length at
least~4. A hole of an oriented graph is any hole of its underlying
graph.  In oriented derived graphs, not only the orientations of holes
are specific, but also they interact in a constrained way.  As
explained in~\cite{Part2} section 6, every hole $H$ in a graph derived
from a Burling tree $T$ has four \emph{special} vertices that we here
describe:

\begin{itemize}
  \item two sources called the
  \emph{antennas}, 
\item one common neighbor of the antennas that is also an ancestor in
  $T$ of all the vertices but the antennas, called the \emph{pivot},
  \item one sink distinct from the pivot, called the \emph{bottom}.
\end{itemize}

Each of the other vertices lie one of the directed paths of $H$ from
an antenna to the bottom. A \emph{subordinate} vertex of a hole
is any vertex distinct from its pivot and antennas (in particular, the
bottom is subordinate and is therefore a descendant of the pivot).  An \emph{extremum} of a hole is any sink or source of it, and a
\emph{transitive vertex} of a hole is any vertex of it with one
in-neighbor and one out-neighbor. So every vertex of a hole is either
a transitive vertex or an extremum.  We often refer to the following
sum up.

\begin{lemma} \label{thm:holes-4-extrema}
  Every hole in an oriented derived graph has exactly two sinks and
  two sources, and the two sources have a common neighbor (which is
  one of the sinks).
\end{lemma}

Now we can summarize some theorems about the interaction of holes in
oriented derived graphs. In what follows, by connecting two
vertices by a path of length 0, we mean identifying the two vertices.
	
A \emph{dumbbell} is a graph made of path $P = x \dots x'$ (possibly
$x = x'$), a hole $H$ that goes through $x$ and a hole $H'$ that goes
through $x'$. Moreover $ V(H) \cap V(P) = \{x\}$,
$ V(H) \cap V(P') = \{x'\}$, $ V(H) \cap V(H') = \{x\} \cap \{x'\} $,
and there are no edges other than the edges of the path and the edges
of the holes. Intuitively, a dumbbell consists of two holes $ H$ and
$ H' $ where one specific vertex of $ H $ is connected by a path $ P $
of length at least~0 to one specific vertex of $ H'$.

\begin{lemma}[\cite{Part2}, Lemma 6.2] \label{thm:dumbbell} Suppose a
  dumbbell with holes $H$, $H'$ and path $P = x \dots x'$ as in the
  definition is the underlying graph of some oriented derived graph
  $G$. Then in $ G$, either $x$ is not a subordinate vertex of $ H $,
  or $ x' $ is not a subordinate vertex of $H'$.
\end{lemma}

A \emph{domino} is a graph made of one edge $xy$ and two holes $H_1$
and $H_2$ that both go through $xy$. Moreover
$V(H_1) \cap V(H_2) = \{x, y\}$ and there are no other edges than the
edges of the holes. Intuitively, a domino consists of two holes
$ H_1 $ and $H_2 $ which have a common edge $ xy$.
	
\begin{lemma}[\cite{Part2}, Lemma 6.3] \label{thm:domino} Suppose a
  domino with holes $H_1$, $H_2$ and edge $xy$ as in the definition is
  the underlying graph of some oriented derived graph $G$. Then for
  some $z\in \{x, y\}$ and some $i\in \{1, 2\}$, $z$ is the pivot of
  $H_i$ and $z$ is a subordinate vertex of $H_{3-i}$.
\end{lemma}

A \emph{theta} is a graph made of three internally vertex-disjoint
paths of length at least~2, each linking two vertices $u$ and $v$
called the \emph{apexes} of the theta (and such that there are no
other edges than those of the paths).  A \emph{long theta} is a theta
such that all the paths between the two apexes of the theta have
length at least~3.

\begin{lemma}[\cite{Part2}, Lemma 6.4]
  \label{lem:theta-pivot}
  If $ G $ is a derived graph whose underlying graph is a long theta
  with apexes $ u $ and $ v $, then exactly one of $ u $ and $ v $ is
  the pivot of every hole of~$ G $.
\end{lemma}

\subsection*{Subdivision and contraction of edges of derived graphs}

Let $G$ be an oriented graph derived from a Burling tree $T$. An arc
$uv$ of $G$ is a \emph{top arc with respect to $T$} if $v$ is the
out-neighbor of $u$ that is closest (in $T$) to the root of $T$. An
arc $uv$ of $G$ is a \emph{bottom arc with respect to $T$} if $v$ is
the out-neighbor of $u$ that is furthest (in $T$) from the root of
$T$.

	\begin{lemma}[\cite{Part2}, Lemma~3.8] \label{thm:subdivision}
          Let $G$ be an oriented graph derived from a Burling tree
          $T$. Any graph obtained from $G$ after performing the
          following operations, any number of times and in any order,
          is an oriented derived graph:
          \begin{enumerate}[label=(\roman*)]
          \item Replacing some bottom arcs $uv$ by a path of length at
            least~1, directed from $u$ to $v$.
          \item Replacing some top arcs $uv$ such that $u$ is a source
            of $G$ by an arc $wv$ and a path of length of length at
            least~1 from $w$ to $u$.
          \end{enumerate}
	\end{lemma}

	The two operations in Lemma~\ref{thm:subdivision} are
        respectively called \emph{subdivision} of arc $ uv $, and
        \emph{top-subdivision} of arc $ uv$.

	\emph{Contraction} of an edge $ uv$ of a (non-oriented) graph $ G $ is to remove the edge $ uv $ and identify the two vertices $ u $ and $ v $. Contraction of an arc in an oriented graph is the contraction of the same edge in the underlying graph.

	\begin{lemma}[\cite{Part2}, Lemma
          3.9] \label{thm:contraction} Let $ G $ be an oriented derived
          graph, and let $ uv $ be an arc such that $ u $ is the only
          in-neighbor of $ v $ and $ v $ is the only out-neighbor of
          $ u $, i.e. $ N^+(u) = \{v\} $ and $ N^-(v) = \{u\}$. Then
          the graph $ G' $ obtained by contracting $ uv $ is also a
          derived graph and the top-arcs (resp.\ bottom-arcs) of $ G $
          but $ uv $ are the top-arcs (resp.\ bottom-arcs) of $ G'$.
	\end{lemma}

	\subsection*{Finding graphs that are not weakly pervasive}
	
	The following lemma enables us to use Burling graphs to find new graphs that are not weakly
        pervasive.
	
	\begin{lemma} \label{theorem:the-method} Let $ H $ be a
          graph. If no subdivision of $ H $ is a Burling
          graph, then $ H $ is not a weakly pervasive graph.
	\end{lemma}
	
	\begin{proof}
          The class of Burling graphs is contained in $ \Forbs (H)
          $. Thus, $\Forbs (H)$ is not $\chi$-bounded. Hence, $ H $ is
          not weakly pervasive.
	\end{proof}
	
%	We remind that Burling graphs and derived graphs are two equal classes of graphs, and to show that a graph is not a Burling graph, we use the derived graph definition to be able to apply the theorems of this section.  

	\section{Complete graphs} \label{section:Kn}

	The graph $ K_3 $ is a weakly pervasive graph, since $ \Forbs(K_3) $ is the class of all forests. In~\cite{Leveque12}, L\'{e}v\^{e}que, Maffray, and Trotignon proved that $ K_4$ is also a weakly pervasive graph. In this section, we prove that $ K_5 $ is not a weakly pervasive graph, and thus is not any $ K_n $ for $ n\geq 5 $. 

	\begin{lemma} \label{lem:subdivisions-of-K5}
		Let $ G $ be a triangle-free subdivision of $ K_5 $. If all the subdivisions of $ K_4 $ in it are of types 2, 3, and 4, then $ G $ has one of the following forms:
		\begin{enumerate}[label=(\roman*)]
			\item type A: edges of a $4$-cycle in $ G $ are not subdivided at all, and any other edge is subdivided at least once. 
			\item type B: edges of a $5$-cycle in $ G $ are not subdivided at all, and any other edge is subdivided at least once.
		\end{enumerate} 
	See Figure~\ref{fig:sd-of-K5}.
	
	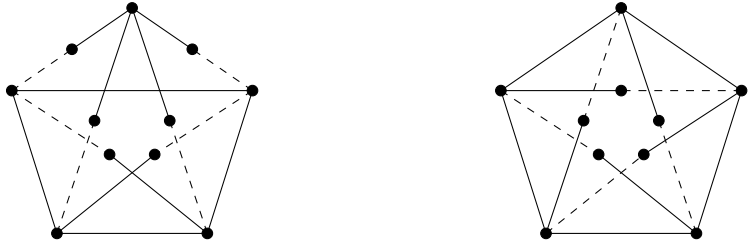
\begin{figure}
		\centering
		\begin{multicols}{2}
			\begin{center}
				\begin{tikzpicture}[scale=.2]
				\filldraw[black] (-5,-5) circle (10pt);
				\filldraw[black] (5,-5) circle (10pt);
				\filldraw[black] (-8,4.5) circle (10pt);
				\filldraw[black] (8,4.5) circle (10pt);
				\filldraw[black] (0,10) circle (10pt);
				\filldraw[black] (-4,7.25) circle (10pt);
				\filldraw[black] (4,7.25) circle (10pt);
				\filldraw[black] (-2.5,2.5) circle (10pt);
				\filldraw[black] (2.5,2.5) circle (10pt);
				\filldraw[black] (-1.5,.25) circle (10pt);
				\filldraw[black] (1.5,.25) circle (10pt);
				\draw[black] (0,10) -- (-4,7.25) ;
				\draw[black] (0,10) -- (4,7.25) ;
				\draw[black] (0,10) -- (-2.5,2.5)  ;
				\draw[black] (0,10) -- (2.5,2.5)  ;
				\draw[black] (-8,4.5) -- (8,4.5) ;
				\draw[black] (-8,4.5) -- (-5,-5);
				\draw[black] (8,4.5) -- (5,-5);
				\draw[black] (-5,-5) -- (5,-5);
				\draw[black] (-5,-5) -- (1.5,.25);	
				\draw[black] (5,-5) -- (-1.5,.25);
				\draw[dashed] (-5,-5) -- (-2.5,2.5);
				\draw[dashed] (5,-5) -- (2.5,2.5);
				\draw[dashed] (-8,4.5) -- (-4,7.25);
				\draw[dashed] (8,4.5) -- (4,7.25);
				\draw[dashed] (-8,4.5) -- (-1.5,.25);
				\draw[dashed] (8,4.5) -- (1.5,.25);
				\end{tikzpicture}
			\end{center} 
			\columnbreak
			\begin{center}
				\begin{tikzpicture}[scale=.2]
				\filldraw[black] (-5,-5) circle (10pt);
				\filldraw[black] (5,-5) circle (10pt);
				\filldraw[black] (-8,4.5) circle (10pt);
				\filldraw[black] (8,4.5) circle (10pt);
				\filldraw[black] (0,10) circle (10pt);
				\filldraw[black] (-2.5,2.5) circle (10pt);
				\filldraw[black] (2.5,2.5) circle (10pt);
				\filldraw[black] (-1.5,.25) circle (10pt);
				\filldraw[black] (1.5,.25) circle (10pt);
				\filldraw[black] (0,4.5) circle (10pt);
				\draw[black] (0,10) -- (-8,4.5) ;
				\draw[black] (0,10) -- (8,4.5) ;
				\draw[dashed] (0,10) -- (-2.5,2.5)  ;
				\draw[black] (0,10) -- (2.5,2.5)  ;
				\draw[black] (-8,4.5) -- (0,4.5) ;
				\draw[dashed] (8,4.5) -- (0,4.5) ;
				\draw[black] (-8,4.5) -- (-5,-5);
				\draw[black] (8,4.5) -- (5,-5);
				\draw[black] (-5,-5) -- (5,-5);
				\draw[dashed] (-5,-5) -- (1.5,.25);	
				\draw[black] (5,-5) -- (-1.5,.25);
				\draw[black] (-5,-5) -- (-2.5,2.5);
				\draw[dashed] (5,-5) -- (2.5,2.5);
				\draw[dashed] (-8,4.5) -- (-1.5,.25);
				\draw[black] (8,4.5) -- (1.5,.25);
				\end{tikzpicture}
			\end{center}
		\end{multicols}
		\caption{Subdivisions of $K_5 $, type A subdivision on the left and type B subdivision on the right. Only dashed edges can be subdivided.} \label{fig:sd-of-K5}
	\end{figure}
	
	\end{lemma}
	
	\begin{proof}
		
		Let $ M= \{ a,b,c,d,e \} $ be the set of vertices of $ G $ of degree $ 4 $. For $ x \in M $, we denote by $ H_x $ the subdivision of $ K_4 $ containing $ M \sm \{x\} $ in $G $. By Lemma~\ref{thm:subdivisions-of-K4}, for all $ x \in M $, $ H_x $ is a type 2, 3, or 4 subdivision of~$ K_4 $. In particular, consider $ H_e $. There are three cases:

		\textit{Case 1.} $H_e $ is of type 2. Without loss of generality, let~$ ac $ and~$ bd $ be the subdivided edges of~$ H_e $. Let $ v \in \{a, b, c, d\} $. If $ ev \in E(G) $, then $ v $ is the center of a wheel in $ G $, a contradiction. Thus, $ ev $ is subdivided, and $ G $ is a type A subdivision of $K_5 $. 
		
%		Thus, $ ab, bc, cd, ad \in E(G) $. Among~$ ae, be, ce, $ and~$de $, at least two of the non-consecutive ones should be subdivided, otherwise, there will be a triangle in~$ G $. So, without loss of generality, we may assume that at least $ ae $ and $ ce $ are subdivided. If~$ be \in E(G) $, then~$ H_d $ is a wheel centered at $ b $, a contradiction. (Remember that a wheel is a subdivision of $ K_4 $ and is not of types 1, 2, or 3.) Similarly, if~$ de \in E(G) $, then~$ H_b $ is a wheel centered at $ d $, a contradiction. Thus, both~$ be $ and~$ de $ are subdivided in~$ G $. So, $ G $ a type~A subdivision of~$ K_5 $.

		\textit{Case 2.} $H_e $ is of type 3. Without loss of generality, let $ ab $, $ ac$, and $ bd $ be the subdivided edges of~$ H_e $. So, $ ad, cd, bc \in E(G) $. If $ ce  $ is not subdivided in $ G $, then $ H_a $ is a wheel centered at $ c $, a contradiction. Thus, $ ce $ is subdivided in~$ G $. Similarly, one can prove that $ de $ must be subdivided. Now, because $ H_d $ must be of type 2, 3, or 4, $ be \in E(G) $. Then, because $ H_c$ must be of type 2, 3, or 4, $ ae \in E(G) $. So, $ G $ is a type B subdivision of $ K_5 $. 
		%Now, consider $ H_c $. We already know that it has three edges which are subdivided, namely $ ad$, $ bd$, and $ be$. So, it is not of type 2.  If $ H_c $ is of type 4, then the other non-subdivided edge of it should be $ ae $, as the two non-subdivided edges of a type 4 subdivision of $ K_4 $ should have a common endpoint. But in this case, $ H_b $ will be a subdivision of $ K_4 $ with exactly two non-subdivided edges which are not adjacent, and thus is not of type 1, 2, or 3, a contradiction. If $ H_c $ is of type 3, it means that $ ae, de \in E(H_c) \subseteq E(G) $. In this case, $ G $ is a type B subdivision of $ K_5 $. 
		
		\textit{Case 3.} $H_e $ is of type 4. Without loss of generality, let $ ab $, $ ac $, $ ad $, and $ bd $ be the subdivided edges of $ H_e $. First of all, $ ce $ must be subdivided, otherwise $ H_a $ will be a wheel centered at $ c $. Secondly, because $ H_b $ should be of type 2, 3, or 4, we must have $ de \in E(G) $. In the same way, because $H_d $ should be of type 2, 3, or 4, we must have $ be \in E(G) $. Finally, $ ae $ must be subdivided, otherwise $H_c $ will be a wheel centered at $ e $. So, in this case, $ G $ is a type A subdivision of $ K_5 $. 
%		
%		\textit{Case 3.} $H_e $ is of type 4. Without loss of generality, let $ ab $, $ ac $, $ ad $, and $ bd $ be the subdivided edges of $ H_e $. Now, consider $ H_c $. Since $ ab $, $ ad $ and $ bd $ are not subdivided, $ H_c $ can only be of type 4. Thus exactly two of $ ae, be, de $ are not subdivided, and one of them is subdivided. Up to symmetry, there are two possibilities for the subdivided edge: $ ae $ and $ be$. If $ be  $ is subdivided, and thus $ ae, de \in E(G)$, then $ ce \notin E(G)$, otherwise $ ecd $ would be a triangle. But then $H_d $ would be a subdivision of $ K_4 $ with two edges with no common endpoints non-subdivided, a contradiction. If $ ae $ is subdivided, and thus $ be, de \in E(G)$, then again $ ce \notin E(G)$, otherwise $ ecd $ would be a triangle. And in this case $ G $ is a type A subdivision of $ K_5 $. 
	\end{proof}
		
	\begin{lemma} \label{lem:typeA-K5}
		If $ G $ is a type A subdivision of $ K_5 $, then it is not a derived graph.
	\end{lemma}
	\begin{proof}
	Let $ \{ a,b,c,d,e \}$ be the set of vertices of degree 4 in $ G $. Without loss of generality, assume that $ ab, bc, cd, da \in E(G) $. The graph $ H $ shown in Figure~\ref{fig:subgraph-typeA-sd-of-K5} is an induced subgraph of~$ G $. 
	\begin{figure}
		\centering
		\vspace*{-.3cm}
		\begin{tikzpicture}[scale=.55]
		\filldraw[black] (-2,-2) circle (3pt) node[anchor=north east] {$d$};
		\filldraw[black] (-2,2) circle (3pt) node[anchor=south east] {$a$};
		\filldraw[black] (2,-2) circle (3pt) node[anchor=north west] {$c$};
		\filldraw[black] (2,2) circle (3pt) node[anchor=south west] {$b$};
		\filldraw[black] (-1,-1) circle (3pt);
		\filldraw[black] (-1,1) circle (3pt);
		\filldraw[black] (1,-1) circle (3pt);
		\filldraw[black] (1,1) circle (3pt);
		\filldraw[black] (0,0) circle (3pt) node[anchor=east] {$e$};
		\draw[black] (2,2) -- (2,-2) -- (-2,-2) -- (-2,2) -- (2,2) -- (1,1);
		\draw[black] (2,-2) -- (1,-1) ;
		\draw[black] (-2,2) -- (-1,1)  ;
		\draw[black] (-2,-2) -- (-1,-1)  ;
		\draw[dashed] (0,0) -- (1,1);
		\draw[dashed] (0,0) -- (1,-1);
		\draw[dashed] (0,0) -- (-1,1);
		\draw[dashed] (0,0) -- (-1,-1);
		\end{tikzpicture} 
		\vspace*{-.25cm}
		\caption{A subgraph of a type A subdivision of $ K_5 $.} \label{fig:subgraph-typeA-sd-of-K5}
	\end{figure}
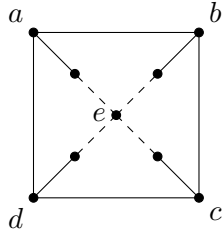
	
	Notice that~$ H $ has no start cutset, it has no vertex of degree 1, and it is not a chandelier (because in $ H$, for every vertex there is a cycle not containing it). Hence, by Theorem~\ref{thm:starcutset-chandelier-degree1}, $ H $ is not a derived graph, and thus, $ G $ is not a derived graph neither.
	\end{proof}
	
	\begin{lemma} \label{lem:typeB-K5}
		If $ G $ is a type B subdivision of $ K_5 $, then it is not a derived graph.
	\end{lemma}
	
	\begin{proof}
		
		Let $ M= \{ a,b,c,d,e \} $ be the set of degree 4 vertices of $ G $. Without loss of generality, we may assume that $ ab, bc, cd, de,  ea \in E(G) $. For $ u,v \in M $, $ u \neq v $, let $ P_{uv} $ denote the degree 2 vertices of the path replacing the edge $ uv $ when subdividing it. In particular $ u,v \notin P_{uv} $. For simplicity in writing, we denote a hole of $ G $ by only naming the vertices of $ M$ in that hole, if there is no confusion.
		
		For the sake of contradiction, suppose that $ G $ is a derived graph. So there is an orientation of $ G $ such that $ G $ is an oriented derived graph. From now on, consider $ G $ with this orientation. We denote the arcs of $ G $ in this orientation by $ A(G) $. 
		
		Consider the hole $ abcdea $ in $ G $. Without loss of
                generality, by Lemma~\ref{thm:holes-4-extrema}, we
                may assume that $ a $ and $ c $ are its sources,
                and~$ b$ and~$ e $ are its sinks.

		Now, consider the hole $ ecde $. Vertex $ d $ is neither a sink nor a source for this hole, and $ c $ also cannot be a sink of it because $ cd \in A(G)$. Therefore, the two sinks of $ ecde $ are among $ P_{ce} \cup \{ e\} $. Call them $ t_1 $ and $ t_2 $, and without loss of generality assume $ t_1 \in P_{ce} $. 
		
		Then, consider the hole $ abca $. For this hole, $ a $ and $ c $ cannot be sinks because $ ab, cd, \in A(G)$, and $ b $ is a sink. So, there is exactly one sink in $ P_{ac} $, call it $ t_3 $. 
		
		Finally, consider the hole $ acea $. Notice that $ t_1 $ and $ t_3 $ are the two sinks of $ acea $. If $ t_2 \in P_{ce} $ then it will be a third sink for $ acea $, a contradiction. If $ t_2 = e $, let $ f \in P_{ce} $ be the neighbor of $ e $ on the subdivided edge between $c $ and $ e $. Then $ fe \in A(G) $, and since we also have $ ae \in A(G) $, then again $ e $ will be a third sink for $ acea $, a contradiction. So $ G $ is not a Burling graph. 
	\end{proof}

	\begin{remark}
		Lemma~\ref{lem:typeA-K5} follows also from Theorem~3.3 of \cite{Chalopin2014}. In fact, by the mentioned theorem, one can see that a type A subdivision of $ K_5 $ is not even a restricted frame graphs, and because Burling graphs are restricted frame graphs, it is not a Burling graphs (or equivalently derived graph) neither.
		
		A type B subdivision of $ K_5 $ on the other hand, can be represented as a restricted frame graph. See Figure~\ref{pic:k5sdB}. 
		
		For the definition of restricted frame graphs, see Definition~2.2 in~\cite{Chalopin2014}, or see Section 6 of~\cite{Part1}.
	\end{remark}

	\begin{figure}
	\begin{center}
		\vspace*{-3cm}
		\includegraphics[width=10cm]{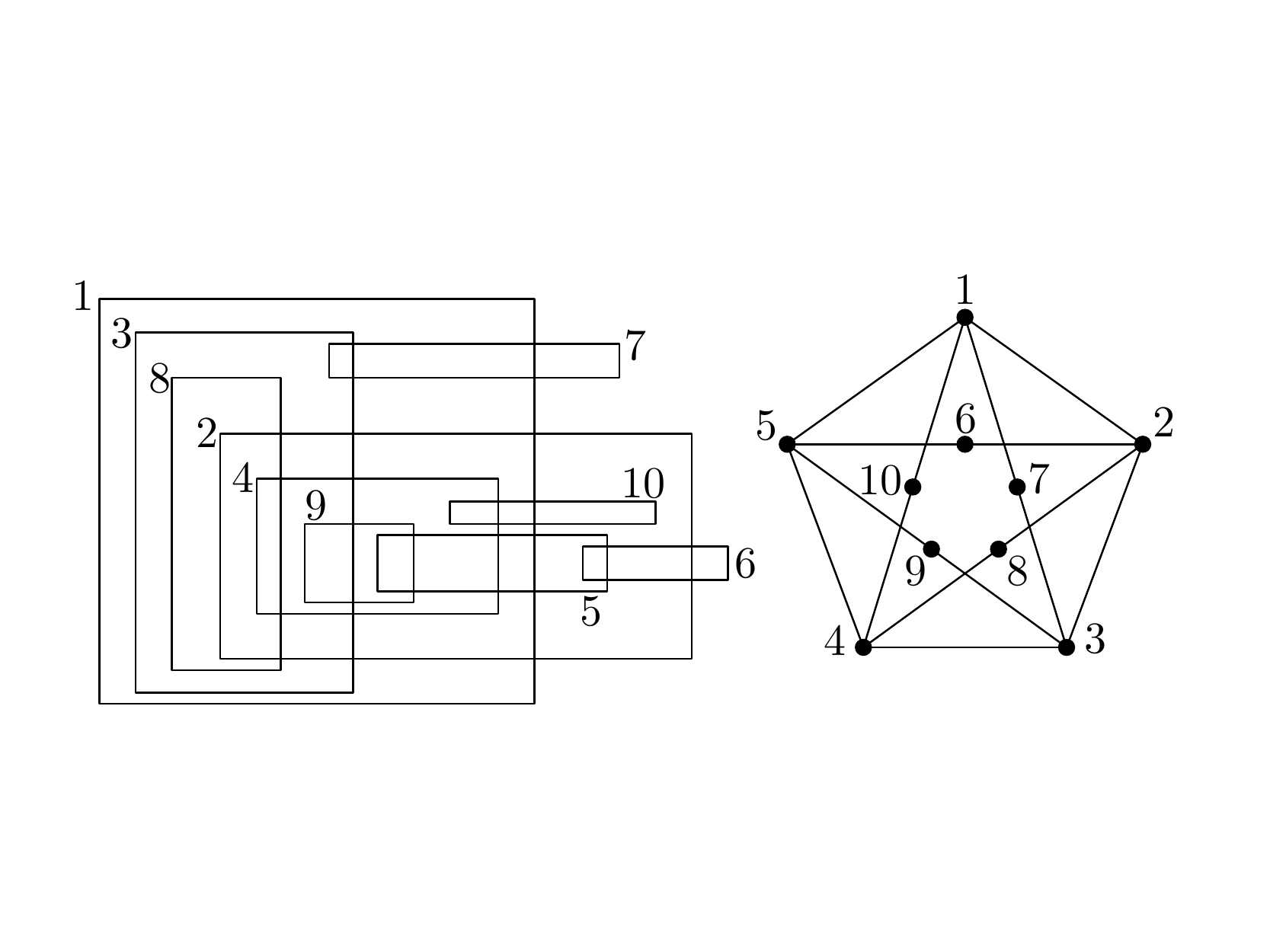}
		\vspace*{-2cm}
		\caption{\footnotesize A type B subdivision of $ K_5 $ and its representation as a restricted frame graph.} \label{pic:k5sdB}
	\end{center}
	\vspace*{-.7cm}
	\end{figure}

	\begin{theorem} \label{thm:K5-not-derived}
		The class of derived graphs includes no subdivision of $ K_5 $.
	\end{theorem}
	
	\begin{proof}
		Let $ G $ be a subdivision of $ K_5 $. If it has a triangle, then it is not a derived graph. So, we may assume that $ G $ is triangle-free. If $ G $ includes a subdivision of $ K_4 $ as an induced subgraph, this subdivision of $K_4 $ must be of type 2, 3, or 4, otherwise, by Lemma~\ref{thm:subdivisions-of-K4}, $ G $ cannot be a derived graph. Thus, by Lemma~\ref{lem:subdivisions-of-K5}, $ G $ is either a type A or a type B subdivision of $ K_5 $. So, the result follows from Lemmas~\ref{lem:typeA-K5} and~\ref{lem:typeB-K5}.
	\end{proof}

		\begin{corollary}
		If $ G $ is a complete graph on $ n $ vertices, where $ n \geq 5 $, then it is is not a weakly pervasive graph.
	\end{corollary}
	\begin{proof}
		For $ n = 5 $, the result follows from Theorem~\ref{thm:K5-not-derived}, using Theorem~\ref{theorem:the-method}. For $ n \geq 5 $ it is enough to notice that $ G $ includes a subdivision of $K_5 $ as an induced subgraph.
	\end{proof}
	
		\begin{figure}
		\centering
		\vspace*{-1.3cm}
		\includegraphics[width=11cm]{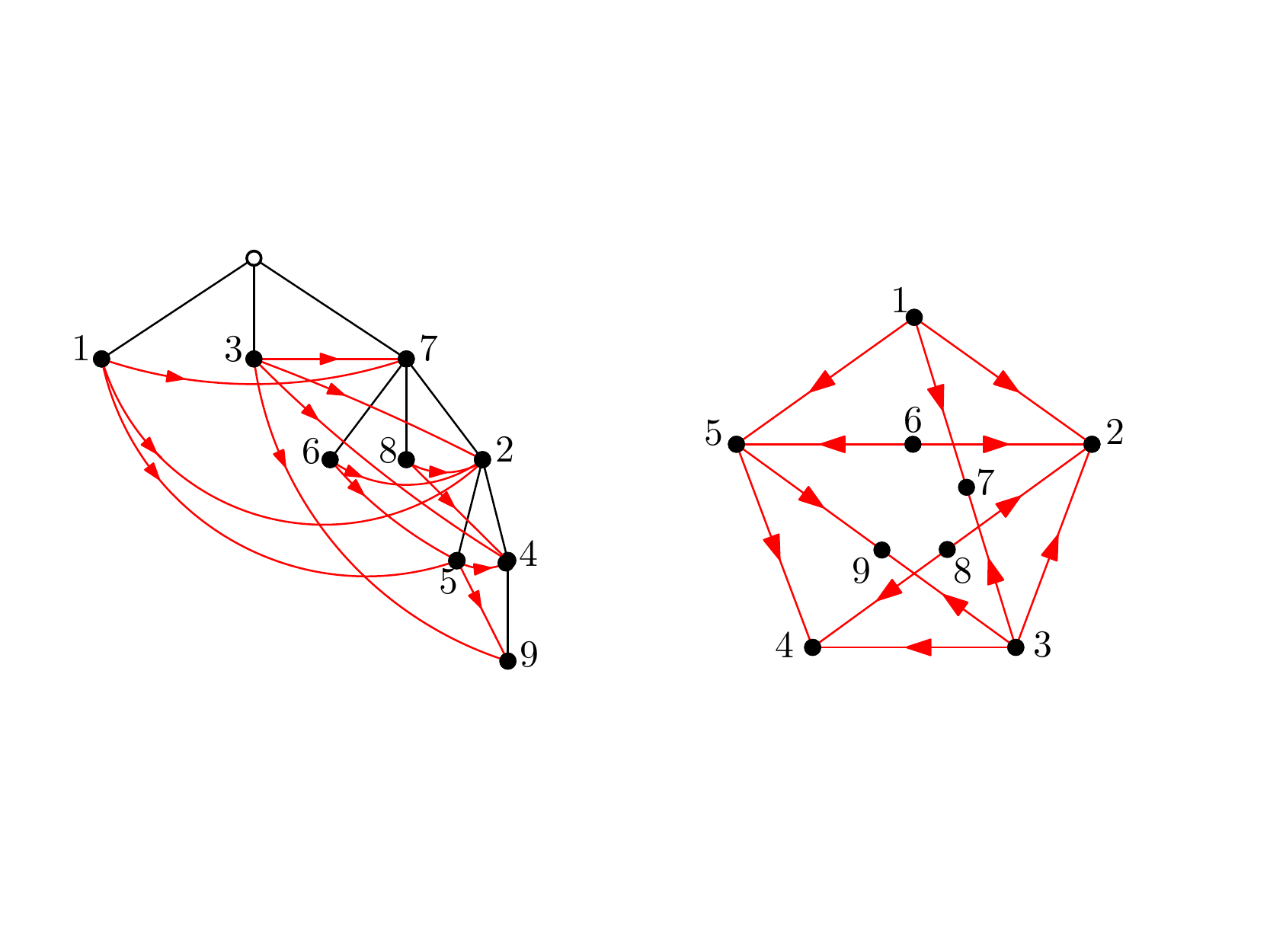}
		\vspace*{-2.3cm}
		\caption{A type B subdivision of $ K_5 $ minus one vertex, shown as a derived graph} \label{pic:K5sdB-minus-1vertex}
	\end{figure}

	\begin{remark}
		Let $ G $ be the graph shown in Figure~\ref{pic:k5sdB}. This graphs is minimally non-Burling: if one removes one vertex of it, it becomes a Burling graph. By deleting a vertex of degree 4 (e.g.\ vertex 1), one obtains a type 3 subdivision of $ K_4 $ which is a Burling graph, and by deleting a vertex of degree 2 (e.g.\ vertex 10) on obtains a graph isomorphic to the underlying graphs of the derived graph represented in Figure~\ref{pic:K5sdB-minus-1vertex}.
	\end{remark}

	\section{Necklace graphs} \label{section:necklace}
	
	We remind that connecting two vertices by a path of length 0 means identifying the two vertices. 
	
	Let $ B_1, B_2, \dots, B_m $, ($ m \geq 2 $), be cycles of length at least~4. For $ 1 \leq i \leq m $, let $ a_i $ and $ b_i $ be two non-adjacent vertices of $ B_i $. For $ 1\leq i \leq m $, connect $ b_i $ and $ a_{i+1} $ by a path of length at least~0 (where $ a_{m+1} = a_1 $). The resulting graph $G$ is called an \emph{$ m$-necklace}. A \emph{necklace} graph is a graph which is an $ m $-necklace for some $ m \geq 2 $. Each $ B_i $ is called a \emph{bead} of $ G $. We say that $ B_i $ is a \emph{short} bead, if $ a_i $ and $b_i $ have a common neighbor. Notice that necklaces are triangle-free graphs. See Figure~\ref{fig:Necklace-example}.

	\begin{figure}
		\centering
		\vspace*{-2cm}
		\includegraphics[width=8cm]{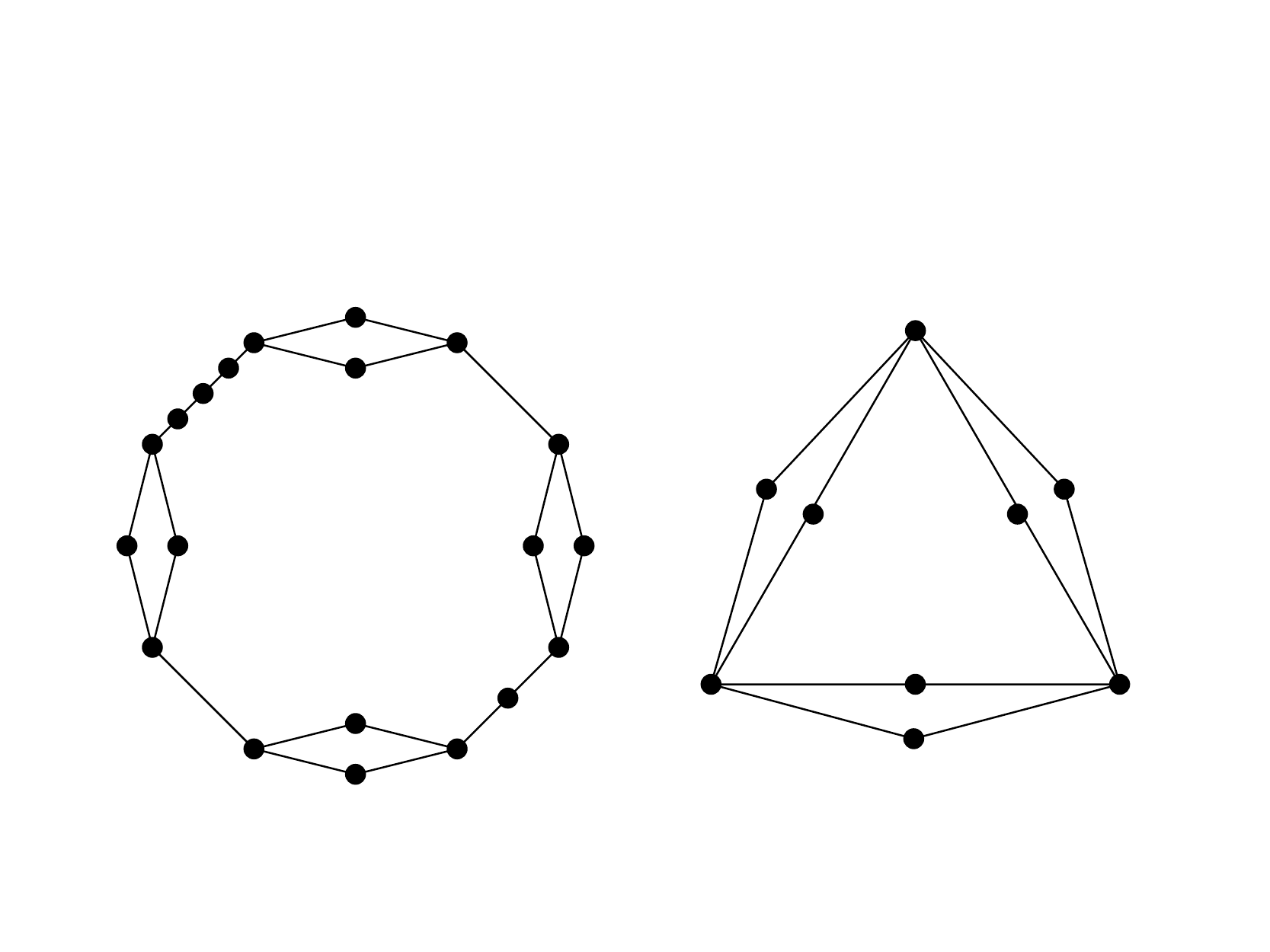}
		\vspace*{-1cm}
		\caption{A 4-necklace (left) and a 3-necklace (right). Any edge of the two graphs can be subdivided.} \label{fig:Necklace-example}
	\end{figure}
	
	In this section, we characterize the necklaces which are Burling graph. Table~\ref{tab:necklaces} shows a summary of the results of this section.

	\begin{table}[h]
		\centering
		\begin{tabular}{|C{.1\textwidth}|C{.4\textwidth}|C{.4\textwidth}|}
			\hline
			 & without star cutset & with star cutset \\
%			 \hline
%			 $ m=1 $ &  \multicolumn{2}{c|}{always Burling graph (Lemma~\ref{lem:1-necklace-is-BG})} \\
			 \hline  
$ m = 2 $ & {Burling graph  $\Leftrightarrow$ the two beads have a common vertex (Lemma~\ref{lem:2-necklace-BG-iff-beads-common-vertex}) } & always	 Burling graph (Lemma~\ref{lem:2-necklace-with-short-bead}) \\
			 \hline
			 $ m = 3 $ & never Burling graph (Lemma~\ref{lem:3-necklace-with-short-bead}) & Burling graph $\Leftrightarrow$ there exists a short bead such that the two other beads have a common vertex (Lemma~\ref{lem:3-necklace-with-short-bead}) \\ % all beads, but at most one short bead, have a common vertex
			 \hline 
			 $ m \geq 4 $ & \multicolumn{2}{c|}{never Burling graph (Lemma~\ref{lem:4-necklaces})} \\
			 \hline
		\end{tabular}
		\caption{$m$-Necklaces and the class of Burling graphs}
		\label{tab:necklaces}
	\end{table}

	\begin{lemma} \label{lem:starcutset-iff-short-bead}
		A necklace graph $ G $ has a star cutset if and only if it has a short bead.
	\end{lemma}
	\begin{proof}
		If $ G $ has a short bead $ B_i $, the common neighbor of $a_i $ and $ b_i $ is the center of a star cutset.
		On the other hand, if $ G $ has no short bead, then it is easy to see that it does not have a star cutset. 
	\end{proof}
	
%	In the proof of the following lemma we need the definition of a \emph{theta}. A theta is any graph, consisting of two vertices $ x $ and  $y $, called the apexes of the theta, and three internally vertex disjoint paths between $x $ and $ y $, and the graph has no edges other that the edges of the three paths. 
%	
%	\begin{lemma} \label{lem:1-necklace-is-BG}
%		Every 1-necklace is a Burling graph.
%	\end{lemma}
%	
%	\begin{proof}
%		A 1-necklace graph $ G $ is either a cycle or a theta. In both cases, it is a chandelier, and so a Burling graph by Theorem~\ref{thm:chandeliers-are-Burling}. 
%	\end{proof}

\begin{lemma}
  \label{lem:2-necklace-common-vertex-beads-isBG}
  \label{lem:2-necklace-BG-iff-beads-common-vertex}
  \label{lem:2-necklace-with-short-bead}
  A 2-necklace graph $ G $ is a derived graph if and only if it has a
  star cutset or its two beads have a common vertex.
\end{lemma}

\begin{proof}
  First, suppose that $ G $ is a derived graph. If it has a star
  cutset, we are done. Otherwise, because it has no vertex of
  degree~1, by Theorem~\ref{thm:starcutset-chandelier-degree1}, it
  should be a chandelier. In particular, there exists a vertex $v $ in
  $ G $ which is contained in all cycles of $ G $. So, the two beads
  of $ G $ both contain $ v $.

  Conversely, if the beads of $ G $ have a common vertex, then $ G $
  is a chandelier (with the pivot being a common vertex of the two
  beads). So, it is a Burling graph by
  Lemma~\ref{thm:chandeliers-are-Burling}. If the beads of $ G $ do
  not have a common vertex, then $ G $ has a start cutset, and thus by
  Lemma~\ref{lem:starcutset-iff-short-bead}, it has a short
  bead. Therefore, $ G $ can be obtained from the underlying graph of
  the graph shown in Figure~\ref{fig:2-necklace-as-BG} (right) by
  subdividing some (possibly none) of the dashed arcs. In
  Figure~\ref{fig:2-necklace-as-BG}, a presentation of the graph on
  the right as a derived graph is shown. Notice that all the dashed
  arcs are either a top-arc starting in a source of $ G $ or a
  bottom-arc of $ G $, so by Lemma~\ref{thm:subdivision}, we may
  subdivide them. So, every 2-necklace with a short bead is a derived
  graph.
  	\begin{figure}
  	\centering
  	\vspace*{-1.5cm}
  	\includegraphics[width=10cm]{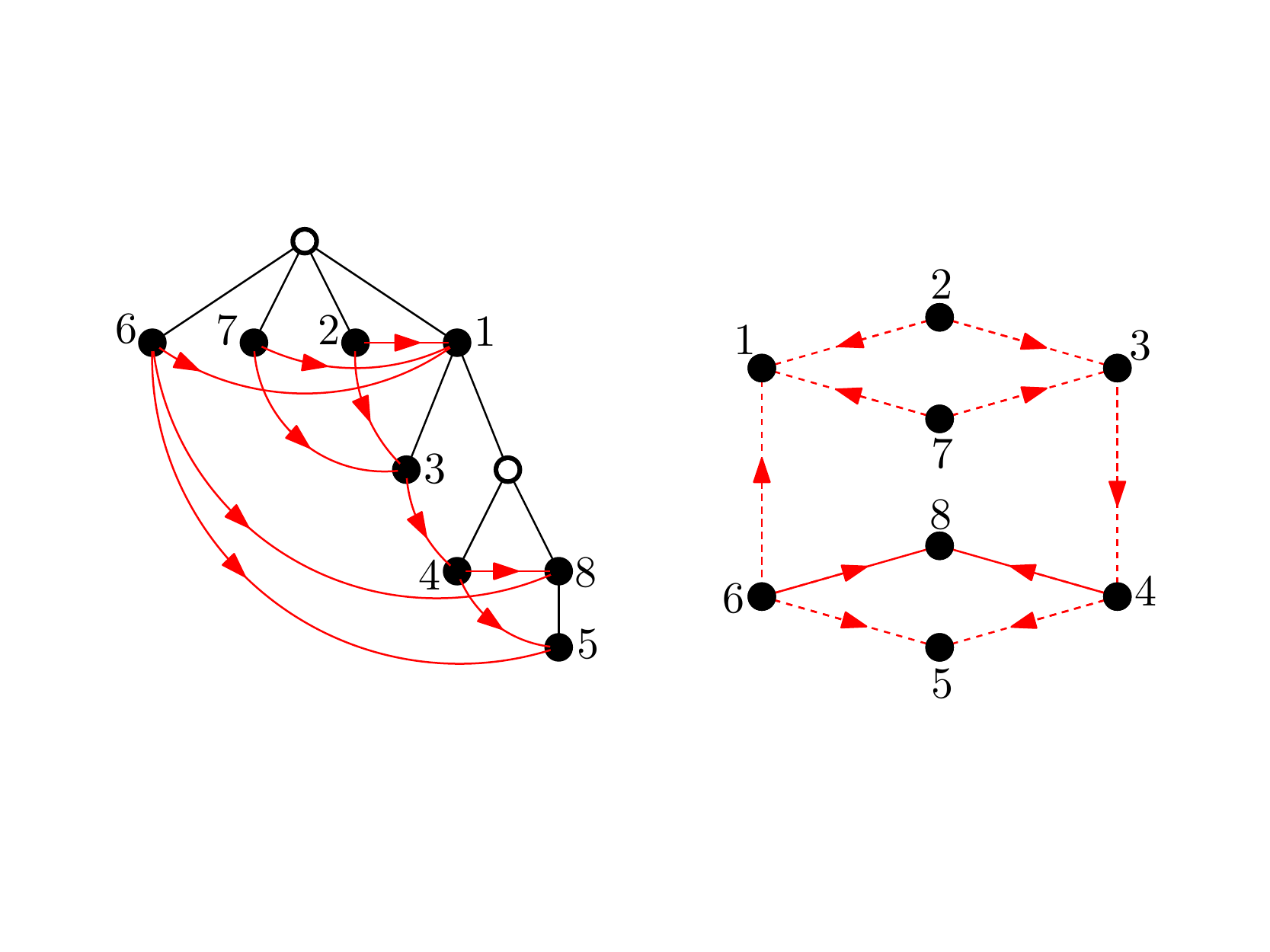}
  	\vspace*{-2cm}
  	\caption{A 2-necklace presented as a derived graph} \label{fig:2-necklace-as-BG}
  \end{figure} 
%  
%	Figure~\ref{fig:2-necklace-as-BG}, a presentation of the graph on the right as a derived graph is shown. Notice that all the dashed arcs are either a top-arc starting in a source of $ G $ or a bottom-arc of $ G $, so by Lemma~\ref{thm:subdivision}, we may subdivide them. So, every 2-necklace with a short bead is a derived graph.
	\end{proof}

	\smallskip
	\begin{lemma} \label{lem:3-necklace-with-short-bead}
		Let $ G $ be a 3-necklace. The graph $ G $ is a derived graph if and only if 
		it has at least one short bead $ B $ and the two other beads have a vertex in common. 
		
		In particular, if $ G $ has no star cutset, then it is not a derived graph.
	\end{lemma}
	\begin{proof}
		First assume that $ G $ is a derived graph and let $B_i $, $ 1\leq i \leq 3 $ be its beads. There is an orientation of $ G $ such that $ G $ with this orientation is an oriented derived graph. For each bead $ B_i $, let $ P_i $ and $ Q_i $ denote the internal vertices of the two paths between $ a_i $ and $ b_i $ in $ B_i $. We call a vertex of a hole an \emph{extremum} if it is a sink or source in that hole. Notice that each $ B_i $ is a hole in $ G $, so by Lemma~\ref{thm:holes-4-extrema} it should have four extrema, and thus there is at least one of $ P_i $ and $ Q_i $, say $P_i $, which contains at least one extremum of $ B_i $. Denote this extremum by $ x_i$. Let $ C $ be the cycle in $ G $ obtained by removing the vertices of $ Q_1 $, $ Q_2 $, and $Q_3 $. Notice that $ C $ is a hole in $ G $, and that $ x_1 $, $ x_2 $, and $x_3 $ are extrema for it. Let $ x_4 $ be the fourth extremum of $ C$. Notice that three of the extrema of $ C $ should be consecutive vertices, and that no two vertices among $ x_1 $, $ x_2 $, and $ x_3 $ are neighbors. So, $ x_4 $ should be the common neighbor of two of  $ x_1 $, $ x_2 $, and $ x_3 $, and $ x_4 $ cannot be inside $ Q_i$'s or $P_i$'s. We assume without loss of generality that $ x_4 $ is a common neighbor of $ x_1 $ and $x_2$. So, $ B_1 $ and $B_2 $ have a common vertex which is $a_2 = b_1 $. Thus, $x_4 $ is the common vertex of $ B_1 $ and $B_2 $. This implies that $ x_1 $ and $ x_2 $ are the antennas of the hole $ C$. So, $ C $ cannot have another antenna. In particular, $ P_3 $ contains no source. Also, $ Q_3 $ contains no source, because we can repeat the same argument by exchanging the role of $P_3 $ and $Q_3 $. Thus, the antennas of $ B_3 $ should be $ a_3 $ and $ b_3 $. Hence, they should have a common neighbor, and therefore $ B_3 $ is a short bead. 
		
		Now, suppose that $ G $ has two beads $ B_1 $ and $ B_2 $ with a common vertex, and a short bead $ B_3 $. In such case an orientation of $ G $ can be obtained from the graph in Figure~\ref{fig:3-necklace-as-BG} by first possibly contracting arc 6-8 or both arcs 6-8 and 7-9, and then subdividing some of the dashed arcs (including 6-8 and 7-9, if they are not contracted). By Lemma~\ref{thm:contraction}, we can contract one or both arcs 6-8 and 7-9, preserving the top-arcs and bottom-arcs. Then, all the dashed edges (including 6-8 and 7-9, if they are not contracted) are top-arcs starting at a source of the graph or bottom-arcs of the graph. So, we can subdivide them as many times as needed. So $ G $ is a derived graph. 
		\begin{figure}
			\centering
			\vspace*{-1cm}
			\includegraphics[width=11cm]{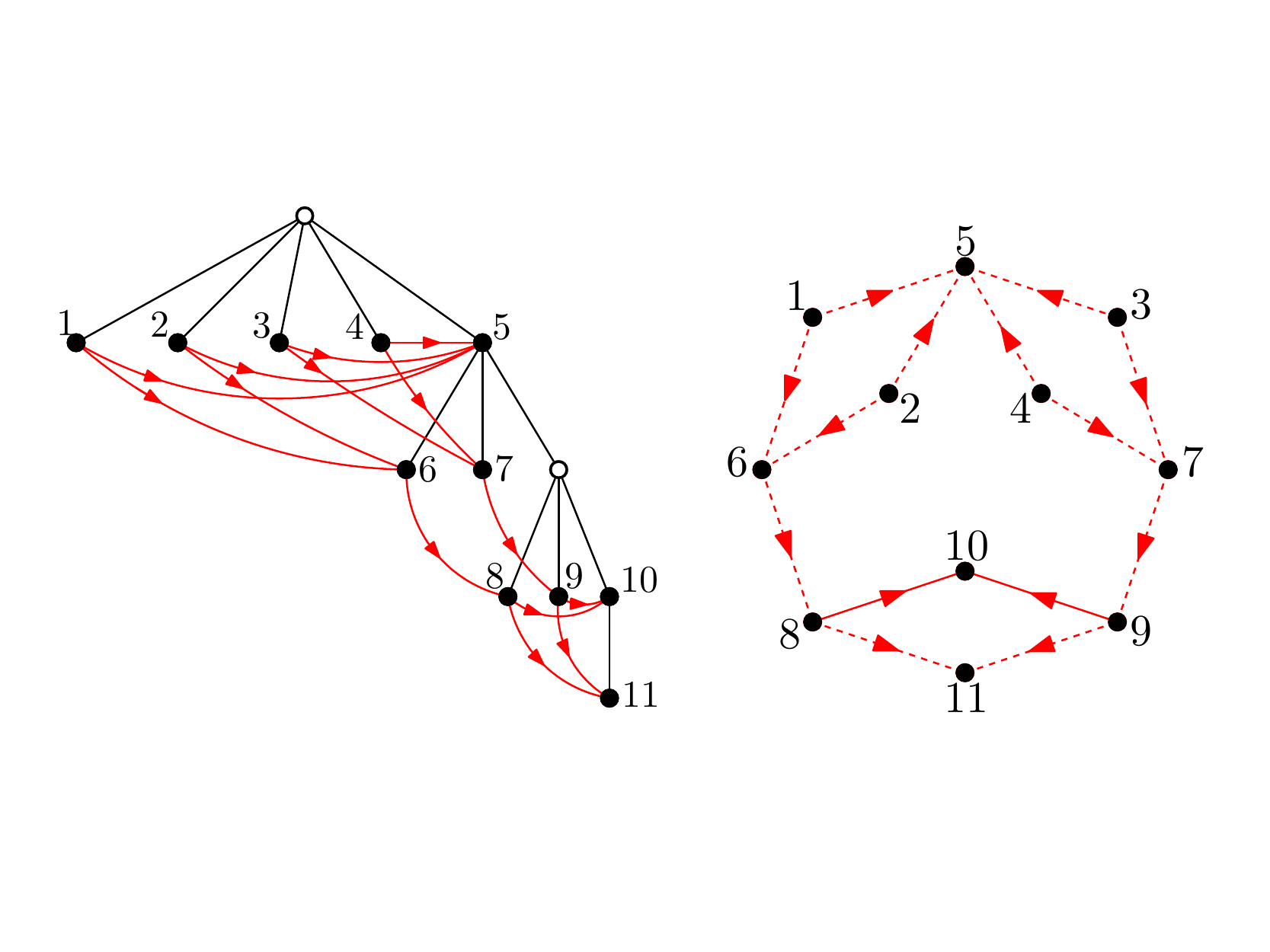}
			\vspace*{-2cm}
			\caption{A 3-necklace presented as a derived graph} \label{fig:3-necklace-as-BG}
		\end{figure}
	
	Finally, notice that if $ G $ has no star cutset, then by Lemma~\ref{lem:starcutset-iff-short-bead} it cannot have a short bead, so it is not a derived graph. 
	\end{proof}
	
	\begin{lemma} \label{lem:4-necklaces}
		Let $ G $ be an $m $-necklace graph. If $ m \geq 4 $, then $ G $ is not a derived graph.
	\end{lemma}
	\begin{proof}
		Assume for the sake of contradiction that $ G $ is a derived graph. So, it is the underlying graph of an oriented derived graph. Consider this orientation on $ G $. Let $ B_1, B_2, \dots, B_m $ be the beads of $ G $. Let $ P_i $ and $ Q_i $ be the internal vertices of the two paths between $ a_i $ and $ b_i $ on the bead $ B_i $, $ 1 \leq i \leq m$. Each $ B_i $ is a hole in $ G $ and thus by Lemma~\ref{thm:holes-4-extrema}, it has four extrema. In particular, at least one of $ P_i $ and $ Q_i $, say $P_i $, has at least one extremum. Let $ C $ be the cycle in $ G $ obtained by removing the vertices of $ Q_i $,$ 1 \leq i \leq m $, and notice that $ C $ is a hole in $ G $. An extremum in $P_i $ for $ B_i $ is also an extremum for $ C $. But $C $ should have exactly four extrema, so $ m \leq 4 $, and hence $ m=4 $, and each $P_i $ has exactly one extremum, and the exrema on $P_i$'s are exactly the extrema of $ C $. Now, notice that three of the extrema of $ C $ should be consecutive vertices, which is not possible because no vertex of $P_i $ is a neighbor of a vertex of $P_j $ if $ i\neq j $. This is a contradiction. So, $G $ is not a Burling graph. 
	\end{proof}
	
	We can summarize all the lemmas above in the following theorem. 
	
	\begin{theorem} \label{thm:main-necklace-thm}
		Let $ G  $ be an $ m $-necklace.
		\begin{enumerate}[label=(\roman*)]
%			\item If $ m = 1$, then $ G $ is a Burling graph.
			\item If $ m=2 $, then $ G $ is a Burling graph if and only if either it has a star cutset or the two beads of $ G $ have a common neighbor. 
			\item If $ m=3 $, then $ G $ is a Burling graph if and only if there exists a short bead such that the two other beads have a common vertex.
			%at least one bead is short and the two other beads have a common vertex.
			\item If $ m\geq 4 $, then $ G $ is not a Burling graph.  
		\end{enumerate}
		See Table~\ref{tab:necklaces}.
	\end{theorem}
	\begin{proof}
		Follows directly from Lemmas~\ref{lem:2-necklace-BG-iff-beads-common-vertex},~\ref{lem:2-necklace-with-short-bead},~\ref{lem:3-necklace-with-short-bead}, and~\ref{lem:4-necklaces}, and the fact that Burling graphs are equal to derived graphs (Theorem~\ref{thm:BG=derived}).
	\end{proof}
	
	\begin{corollary}
		Let $ G $ be an $ m $-necklace graph. If any of the followings happens, then $ G  $ is not a weakly pervasive graph:
		\begin{enumerate}[label=(\roman*)]
			\item $ m=2 $, $ G $ has no star cutset, and the beads do not share a vertex,
%			\item $ m=3 $ and $ G $ has no start cutset,
			\item $ m =3 $ and $ G $ has no star cutset, 
			\item $ m =3 $ and for every short bead $ B $ of $ G $, the two other beads have no common vertex. %and for every vertex $ v $ of $ G $, there is a bead other than $ B $ which do not contain $ v $,
			\item $ m \geq 4 $.
		\end{enumerate}
	\end{corollary}
	\begin{proof}
		Let $ G $ be in one of the above four forms. Notice that a subdivision of $ G $ will also be in one of the above forms. So, by Theorem~\ref{thm:main-necklace-thm}, neither $ G $ nor any of its subdivisions ae Burling graph. Notice that we have used the fact that having a short bead and having a star cutset are equivalent in necklaces (Lemma~\ref{lem:starcutset-iff-short-bead}). So, $ G $ is not a weakly pervasive graph.
	\end{proof}
	
	\begin{remark}
		\begin{figure}
			\centering
			\vspace*{-2cm}
			\includegraphics[width=8cm]{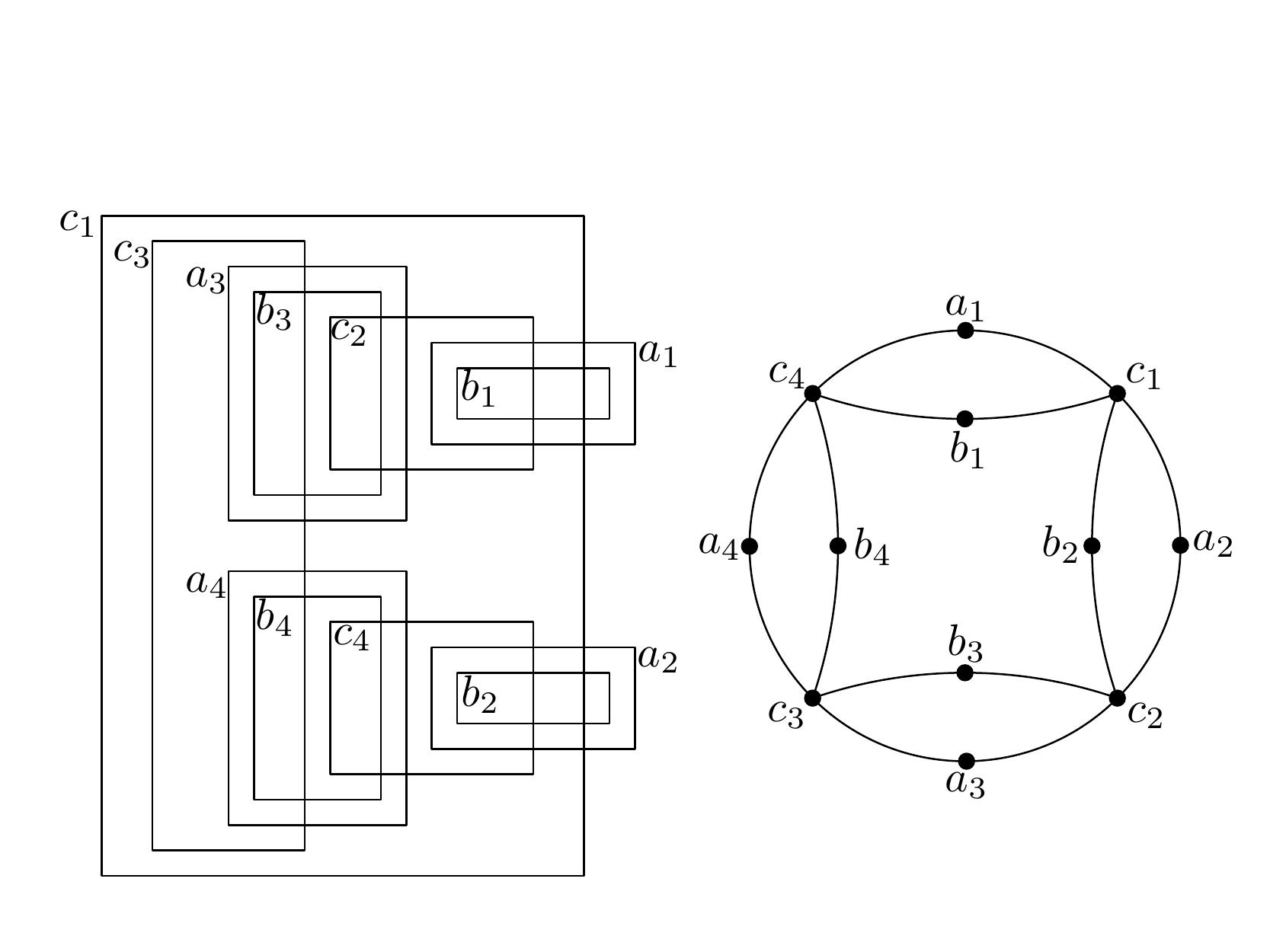}
			\caption{A $ 4$-necklace, presented as a frame graph} \label{fig:4-necklace-as-frame-graph}
		\end{figure}
		Let $ G $ be a necklace. If $ G $ does not have a star cutset, then $ G $ is not even a restricted frame graph. This follows from Theorem~3.3 of~\cite{Chalopin2014}. But if $ G $ has a star cutset, then it might be a restricted frame graph. See Figure~\ref{fig:4-necklace-as-frame-graph} for an example of a non-Burling necklace which is a restricted frame graph.
	\end{remark}
	
	\section{Dumbbells} \label{section:dumbbells}

	Let $ G $ be a non-oriented derived graph. We call a vertex $ v $ a \emph{global subordinate} vertex of $ G $ if for any oriented derived graph $\tilde G $ for which $ G $ is the underlying graph, $ v $ is a subordinate vertex of some hole in $ \tilde G $.

	Given two graphs $ G_1 $ and $G_2 $ and two vertices $ x_1 \in V(G_1) $ and $x_2 \in V(G_2) $, one can build a graph $ D $ as follows: first take the disjoint union of $G_1 $ and $G_2 $ and then connect $x_1 $ and $x_2 $ by a path of length at least~0, to obtain $ D$. Any graph $ D $ built as above is called a \emph{dumbbell} of $ G_1 $ and $G_2 $ with respect to $x_1 $ and $x_2 $. See Figure~\ref{fig:dumbbells} for some examples.

	\begin{lemma} \label{lem:general-dumbbell}
	Let $ G_1 $ and $G_2 $ be two derived  graphs, and let $ x_1 \in V(G_1) $ and $x_2 \in V(G_2) $. If $ x_i $ is a global subordinate vertex of $G_i $, for $ i=1,2$, then any dumbbell $ D $ of $G_1 $ and $G_2 $ with respect to $x_1 $ and $x_2 $ is not a derived graph.  
	\end{lemma}
	
	\begin{proof}
		Assume that $ D $ is a derived graph. By definition of global subordinate vertex, for $ i = 1,2 $, there is a hole $ H_i $ in $G_i $ for which $ x_i $ is a subordinate vertex. So, the dumbbell built by holes $ H_1 $ and $H_2$ and the path between $ x_1 $ and $x_2 $ is an induced subgraph of $ D $, and thus is a derived graph, a contradiction with Lemma~\ref{thm:dumbbell}.
	\end{proof}

	\begin{remark}
	Notice that if one of $ G_1 $ or $G_2 $ is not a derived graph, then the dumbbell of $G_1 $ and $G_2 $ with respect to any two vertices is not a derived graph, because Burling graphs are closed under taking induced subgraphs.
	\end{remark}

	In this section, Lemma~\ref{lem:general-dumbbell} is essentially what enables us to find a new family of graphs that are not weakly pervasive, all of which have vertex cuts. 
	
	\begin{figure}
		\centering
		\vspace*{-1cm}
		\includegraphics[width=8cm]{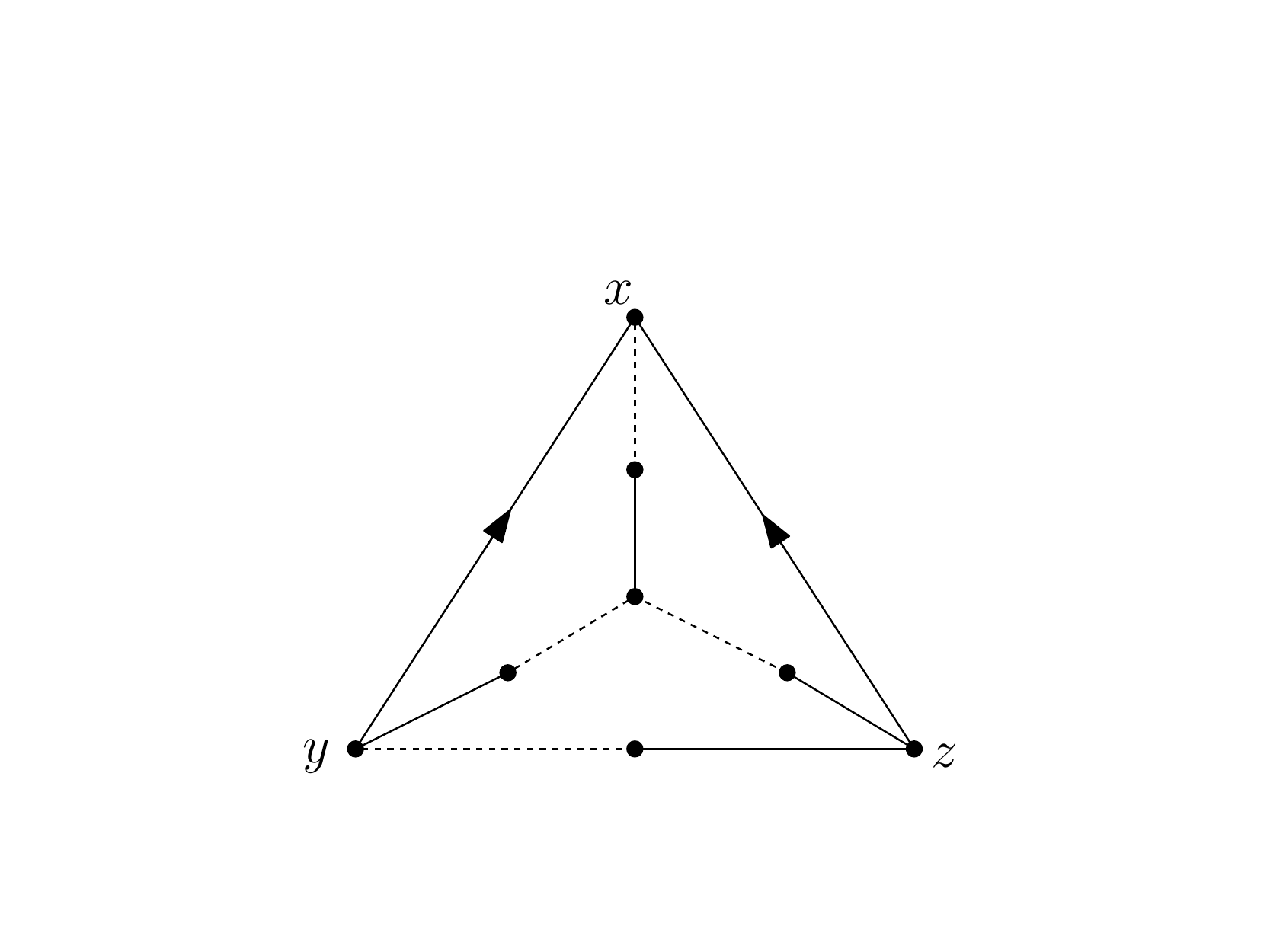}
		\vspace*{-1.25cm}
		\caption{The vertex $ x $ is the unique center of an in-star cutset in $ G $.} \label{Fig:unique-center}
	\end{figure}

	\begin{lemma} \label{lem:unique-in-star-in-K4-subd}
		Let $ G $ be an oriented derived graph whose underlying graph is a type 4 subdivision of $ K_4 $. If $ x $ is the common end-point of the two non-subdivided arcs, then $ x $ is the unique center of an in-star cutset in $ G $. In particular, the two non-subdivided arcs are oriented toward $ x $. See Figure~\ref{Fig:unique-center}.
	\end{lemma}
	
	\begin{proof}		
		First, $ G $ has an in-star cutset due to Theorem~\ref{thm:starcutset-chandelier-degree1} because it is not a chandelier and has no vertex of degree~1. 
		
		Moreover, if a vertex $ v $ is the center of an in-star cutset in $ G $, then $ v$ is the center of a star cutset in the underlying graph of $ G $. Since any vertex other than $x $ cannot be the center of a star cutset in the underlying graph of $ G $, $ x $ is the unique center of an in-star cutset in $ G $.%, and this proves (i).
	\end{proof}

	\begin{lemma} \label{lem:K-4-pre-subordinate} \label{lem:K-4-subordinate}
		Let $ G $ be a type 4 subdivision of $ K_4 $. Let $ x  $ be the common end-point of the two non-subdivided edges of $ G $, and let $ y $ and $ z $ be its degree 3 neighbors. For any subdivision $ G^* $ of $ G $ which is a derived graph,
		$ yx, zx \in E(G^*) $, 
		and for any $ w \in \{y,z\} $, $ x $ is the pivot of exactly one of the two holes going through $ xw $, and is a subordinate vertex of the other. In particular, $ x $ is a global subordinate vertex of $ G^* $.
	\end{lemma}
	\begin{proof}
          If either of $ xy $ or $ xz $ are not edges of~$ G^* $, then
          by Lemma~\ref{thm:subdivisions-of-K4}, $ G^* $ is not a
          derived graph and there is nothing to prove. So, suppose that
          $ yx, zx \in E(G^*) $, and hence $ G^* $ is a derived graph,
          meaning that it is the underlying graph of an oriented
          derived graph. Consider this orientation on~$ G^* $. By
          Lemma~\ref{lem:unique-in-star-in-K4-subd}, $ x $ is the only
          center of an in-star cutset in~$ G^*$.
		%Notice that $ G^*$ is not a chandelier (there is no vertex common among all cycles), and it does not have a vertex of degree 1, so by Theorem~\ref{thm:starcutset-chandelier-degree1}, $ G^* $ must have an in-start cutset. One can see that the only vertex in $ G^* $ that can be the center of an in-star cutset is $ x $.
		So, $ x $ is not the antenna of any hole in $ G^* $. Consider the two holes passing through the arc $ xw $ in $ G^*$ and call them $ H_1 $ and $H_2 $. They form a domino, so by Lemma~\ref{thm:domino}, for some $ u \in \{x,w\} $ and for some $ i \in \{1,2\} $, 
		$u$ is the pivot
		of $ H_i$ and $u $ is a subordinate vertex of $
                H_{3-i} $. Because the arc $ wx $ is oriented from $ w
                $ to $ x $, then $ w $ cannot be the pivot of any of
                the two holes. Thus in $G^*$, the vertex $ x $ is the
                pivot of one of the two holes, and the subordinate
                vertex of another.
\end{proof}
%
%	\begin{lemma}\label{lem:K-4-subordinate}
%		Let $ G $ be a type 4 subdivision of $ K_4 $. If $ x $ is the common end-point of the two non-subdivided edges of $ G $, then $ x $ is a global subordinate vertex of any subdivision of $ G $ which is a derived graph.
%	\end{lemma}
%	\begin{proof}
%		Follows immediately from Lemma~\ref{lem:K-4-pre-subordinate}.
%	\end{proof}

\begin{lemma} \label{lem:theta-subordinate} Let $ G $ be a long theta
  with apexes $ u $ and $v $. If $ x $ is a vertex of degree 2 in
  $ G $ such that its two neighbors are also of degree 2, then $ x $
  is a global subordinate vertex in any subdivision of $ G $.
\end{lemma}
              
\begin{proof}
  By Lemma~\ref{lem:theta-pivot}, for any derived graph $\tilde G $
  for which $ G $ is an underlying graph, there is $ w \in \{u,v\} $
  such that $ w $ is the pivot of all holes of $ \tilde G $, and thus
  the antennas of all holes are among the neighbors of $ w $. By
  assumptions, $ x \neq w $ and is not one of its neighbors. Thus,
  $ x $ is a global subordinate vertex of~$ G $.
\end{proof}
	
\begin{theorem} \label{thm:non-Burling-dumbbells} Let $ G_1 $ be a
  type 4 subdivision of $K_4 $, and let $ x_1 $ be the common
  end-point of its two non-subdivided edges. Let $ G_2 $ be a long
  theta, and let $ x_2 $ be a vertex of degree 2 in $ G_2 $ whose
  neighbors are also of degree 2. Choose $i $ and $ j $ in $\{1,2\}$
  ($i $ and $ j $ can be possibly equal). If $D $ is a dumbbell of
  $ G_i $ and $G_j $ with respect to $ x_i$ and $ x_j $, then no
  subdivision of $ D $ is a Burling graph. In particular, $ D $ is
  not weakly pervasive. See Figure~\ref{fig:dumbbells}.
	\end{theorem}
	\begin{proof}
	If a subdivision of $ G_i $ is not a Burling graph, then neither is the dumbbell containing it. Otherwise, the result follows from Lemma~\ref{lem:K-4-subordinate} and Lemma~\ref{lem:theta-subordinate} by applying Lemma~\ref{lem:general-dumbbell}.
	
	The second part follows from Lemma~\ref{theorem:the-method}.
	\end{proof}	

%	\begin{corollary}
%		If $ D $ is a dumbbell as defined in Theorem~\ref{thm:non-Burling-dumbbells}, then $ D $ is non-weakly pervasive. See Figure~\ref{fig:dumbbells}.
%	\end{corollary}
%	\begin{proof}
%		It is enough to notice that any subdivision of such dumbbell is a dumbbell of the same type, and then apply Theorem~\ref{thm:non-Burling-dumbbells}.
%	\end{proof}

	\begin{figure}
		\vspace*{-.5cm}
				\begin{center}
			\begin{tikzpicture}[scale=.4] % two K4
			% the right K4
			\filldraw[black] (9,0) circle (5pt);
			\filldraw[black] (9,4) circle (5pt);
			\filldraw[black] (9,-4) circle (5pt);
			\filldraw[black] (8,2) circle (5pt);
			\filldraw[black] (8,-2) circle (5pt);
			\filldraw[black] (7,0) circle (5pt);
			\filldraw[black] (5,0) circle (5pt);
			\filldraw[black] (3,0) circle (5pt);
			\draw[black] (3,0) -- (9,4) -- (9,-4) -- (3,0);
			\draw[black] (3,0) -- (7,0) -- (9,4);
			\draw[black] (9,-4) -- (7,0);
			% the left K4
			\filldraw[black] (-9,0) circle (5pt);
			\filldraw[black] (-9,4) circle (5pt);
			\filldraw[black] (-9,-4) circle (5pt);
			\filldraw[black] (-8,2) circle (5pt);
			\filldraw[black] (-8,-2) circle (5pt);
			\filldraw[black] (-7,0) circle (5pt);
			\filldraw[black] (-5,0) circle (5pt);
			\filldraw[black] (-3,0) circle (5pt);
			\draw[black] (-3,0) -- (-9,4) -- (-9,-4) -- (-3,0);
			\draw[black] (-3,0) -- (-7,0) -- (-9,4);
			\draw[black] (-9,-4) -- (-7,0);
			% the connection
			\draw[dotted] (-3,0) -- (3,0);
			\end{tikzpicture}\\
			\vspace*{1cm}
			\begin{tikzpicture}[scale=.4] % one K4 one theta
			% the right K4
			\filldraw[black] (9,0) circle (5pt);
			\filldraw[black] (9,4) circle (5pt);
			\filldraw[black] (9,-4) circle (5pt);
			\filldraw[black] (8,2) circle (5pt);
			\filldraw[black] (8,-2) circle (5pt);
			\filldraw[black] (7,0) circle (5pt);
			\filldraw[black] (5,0) circle (5pt);
			\filldraw[black] (3,0) circle (5pt);
			\draw[black] (3,0) -- (9,4) -- (9,-4) -- (3,0);
			\draw[black] (3,0) -- (7,0) -- (9,4);
			\draw[black] (9,-4) -- (7,0);
			% the left theta
			\filldraw[black] (-3,0) circle (5pt);
			\filldraw[black] (-4.5,2) circle (5pt);
			\filldraw[black] (-4.5,-2) circle (5pt);
			\filldraw[black] (-6,-4) circle (5pt);
			\filldraw[black] (-6,1.5) circle (5pt);
			\filldraw[black] (-6,-1.5) circle (5pt);
			\filldraw[black] (-6,4) circle (5pt);
			\filldraw[black] (-9,2) circle (5pt);
			\filldraw[black] (-9,-2) circle (5pt);
			\draw[black] (-3,0) -- (-6,4) -- (-9,2) -- (-9,-2) -- (-6,-4) -- (-3,0);
			\draw[black] (-6,4) -- (-6,-4) ;
			% the connection
			\draw[dotted] (-3,0) -- (3,0);
			\end{tikzpicture} \\
			\vspace*{1cm}
			\begin{tikzpicture}[scale=.4] % two theta
			% the right theta
			\filldraw[black] (3,0) circle (5pt);
			\filldraw[black] (4.5,2) circle (5pt);
			\filldraw[black] (4.5,-2) circle (5pt);
			\filldraw[black] (6,-4) circle (5pt);
			\filldraw[black] (6,1.5) circle (5pt);
			\filldraw[black] (6,-1.5) circle (5pt);
			\filldraw[black] (6,4) circle (5pt);
			\filldraw[black] (9,2) circle (5pt);
			\filldraw[black] (9,-2) circle (5pt);
			\draw[black] (3,0) -- (6,4) -- (9,2) -- (9,-2) -- (6,-4) -- (3,0);
			\draw[black] (6,4) -- (6,-4) ;
			% the left theta
			\filldraw[black] (-3,0) circle (5pt);
			\filldraw[black] (-4.5,2) circle (5pt);
			\filldraw[black] (-4.5,-2) circle (5pt);
			\filldraw[black] (-6,-4) circle (5pt);
			\filldraw[black] (-6,1.5) circle (5pt);
			\filldraw[black] (-6,-1.5) circle (5pt);
			\filldraw[black] (-6,4) circle (5pt);
			\filldraw[black] (-9,2) circle (5pt);
			\filldraw[black] (-9,-2) circle (5pt);
			\draw[black] (-3,0) -- (-6,4) -- (-9,2) -- (-9,-2) -- (-6,-4) -- (-3,0);
			\draw[black] (-6,4) -- (-6,-4) ;
			% the connection
			\draw[dotted] (-3,0) -- (3,0);
			\end{tikzpicture}
		\end{center}
	\caption{Some graphs that are not weakly pervasive. One can subdivide any edge, or contract the dotted edges, and still have a graph that is not weakly pervasive.} \label{fig:dumbbells}
	\end{figure}
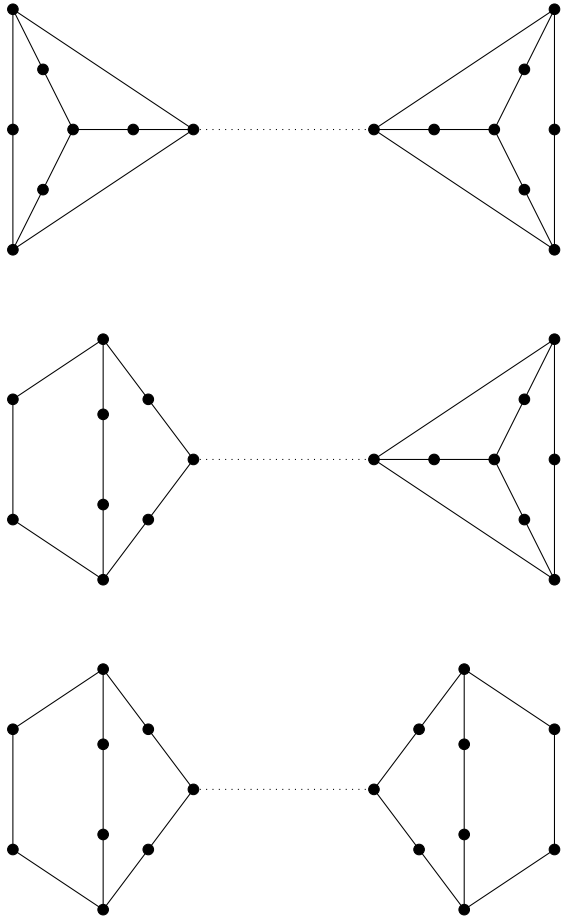
			\begin{figure}
		\centering
		\vspace*{-1.8cm}
		\includegraphics[width=10cm]{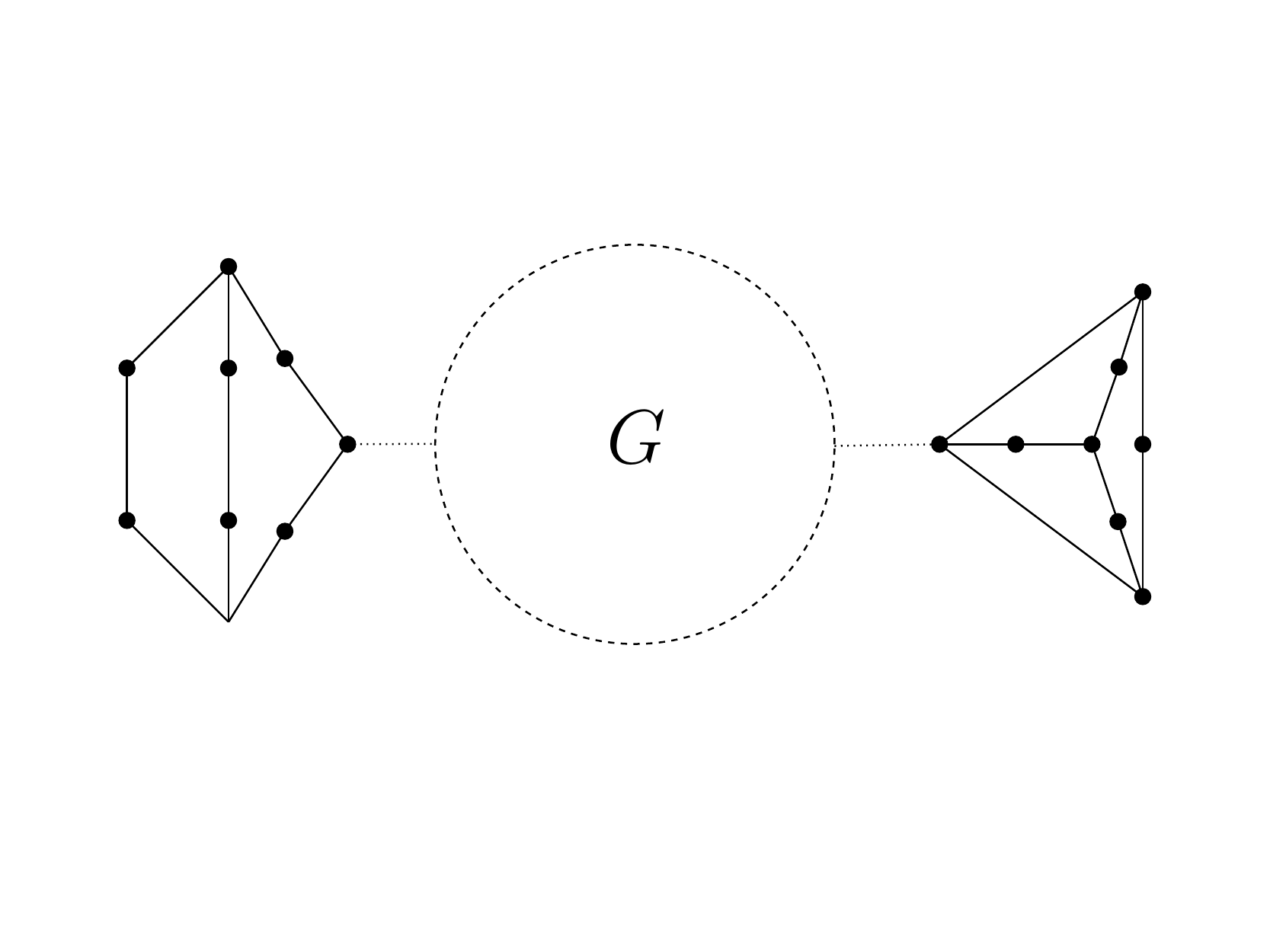}
		\vspace*{-2.3cm}
		\caption{More graph that are not weakly pervasive graphs based on dumbbells. Here, $ G $ is any arbitrary connected graph.} \label{fig:more-for-dumbbell-examples}
	\end{figure}
	
	\begin{remark}
		Notice that the technique used in this section can help us to obtain more graphs that are not weakly pervasive. In particular, any graph having an induced subgraph isomorphic to $ D $ as described in Theorem~\ref{thm:non-Burling-dumbbells}, is not weakly pervasive. Thus the graph in Figure~\ref{fig:more-for-dumbbell-examples}, in which $ G $ is any arbitrary connected graph, has $D $ as an induced subgraph, and is not weakly pervasive. 
%		\begin{figure}
%			\centering
%			\begin{tikzpicture}[x=.5cm,y=.2cm] % one K4 one theta
%			% the right K4
%			\filldraw[black] (9,0) circle (5pt);
%			\filldraw[black] (9,4) circle (5pt);
%			\filldraw[black] (9,-4) circle (5pt);
%			\filldraw[black] (8,2) circle (5pt);
%			\filldraw[black] (8,-2) circle (5pt);
%			\filldraw[black] (7,0) circle (5pt);
%			\filldraw[black] (5,0) circle (5pt);
%			\filldraw[black] (3,0) circle (5pt);
%			\draw[black] (3,0) -- (9,4) -- (9,-4) -- (3,0);
%			\draw[black] (3,0) -- (7,0) -- (9,4);
%			\draw[black] (9,-4) -- (7,0);
%			% the left theta
%			\filldraw[black] (-3,0) circle (5pt);
%			\filldraw[black] (-4.5,2) circle (5pt);
%			\filldraw[black] (-4.5,-2) circle (5pt);
%			\filldraw[black] (-6,-4) circle (5pt);
%			\filldraw[black] (-6,1.5) circle (5pt);
%			\filldraw[black] (-6,-1.5) circle (5pt);
%			\filldraw[black] (-6,4) circle (5pt);
%			\filldraw[black] (-9,2) circle (5pt);
%			\filldraw[black] (-9,-2) circle (5pt);
%			\draw[black] (-3,0) -- (-6,4) -- (-9,2) -- (-9,-2) -- (-6,-4) -- (-3,0);
%			\draw[black] (-6,4) -- (-6,-4) ;
%			% the connection
%			\draw[dotted] (-3,0) -- (-1.5,0);
%			\draw[dotted] (1.5,0) -- (3,0);
%			\draw (0,0) circle (1.5) node[] {$G$};
%			\end{tikzpicture}
%			\caption{More non-weakly pervasive graphs based on dumbbells. Here, $ G $ is any arbitrary graph.} \label{fig:more-for-dumbbell-examples}
%		\end{figure}
	\end{remark}

\section{Miscellaneous} \label{section:Misc}

	\subsection*{First example}
	
	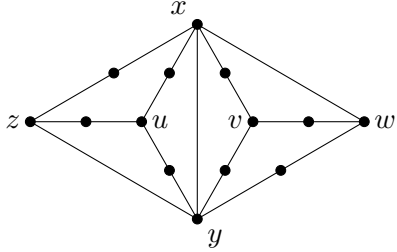
\begin{figure}
		\begin{center}
			\begin{tikzpicture}[scale=.37]
			\filldraw[black] (0,3.5) circle (5pt) node[anchor=south east] {$x$};
			\filldraw[black] (0,-3.5) circle (5pt) node[anchor=north west] {$y$};
			\filldraw[black] (-6,0) circle (5pt) node[anchor=east] {$z$};
			\filldraw[black] (6,0) circle (5pt) node[anchor=west] {$w$};
			\filldraw[black] (-2,0) circle (5pt) node[anchor=west] {$u $};
			\filldraw[black] (2,0) circle (5pt) node[anchor=east] {$v$};
			\filldraw[black] (-4,0) circle (5pt);
			\filldraw[black] (4,0) circle (5pt);
			\filldraw[black] (-1,1.75) circle (5pt);
			\filldraw[black] (1,1.75) circle (5pt);
			\filldraw[black] (-1, -1.75) circle (5pt);
			\filldraw[black] (1, -1.75) circle (5pt);
			\filldraw[black] (-3, 1.75) circle (5pt);
			\filldraw[black] (3,-1.75) circle (5pt);
			\draw[black] (0,3.5) -- (0,-3.5);
			\draw[black] (0,3.5) -- (2,0);
			\draw[black] (0,-3.5) -- (2,0);
			\draw[black] (6,0) -- (2,0);
			\draw[black] (0,3.5) -- (-2,0);
			\draw[black] (0,-3.5) -- (-2,0);
			\draw[black] (-6,0) -- (-2,0);
			\draw[black] (-6,0) -- (0,-3.5);
			\draw[black] (-6,0) -- (0,3.5);
			\draw[black] (6,0) -- (0,-3.5);
			\draw[black] (6,0) -- (0,3.5);
			\end{tikzpicture}
		\end{center}
	\vspace*{-.3cm}
		\caption{A graph that is not weakly pervasive. Theorem~\ref{thm:TwinK4-no-BG}.} \label{fig:twin-K4}
	\end{figure}
	
	\begin{theorem} \label{thm:TwinK4-no-BG}
		Let $ G $ be the graph shown in Figure~\ref{fig:twin-K4}. No subdivision of $ G $ is a derived graph. In particular, $ G $ is not weakly pervasive. 
	\end{theorem}
	
	\begin{proof}
		Let $ G^* $ be a subdivision of $G $. The graph $ G^* $ contains two subdivisions of $ K_4 $ as induced subgraphs: one whose degree 4 vertices are exactly $ \{u, x, y, z\} $ and does not contain $ v $ and $ w $, which we denote by $H_l $, and the other whose degree 4 vertices are exactly $\{ v, x, y, w\} $ and does not contain $ u $ and $ z$, which we denote by $ H_r $.
		
		If in constructing $ G^* $, any of the edges $ xy $, $zy $, or $ xw $ is subdivided, then $ G^* $ contains a subdivision of $K_4 $ which, by Lemma~\ref{thm:subdivisions-of-K4}, is not a Burling graph, so $ G^* $ is not a Burling graph. Hence, we may assume that  $ xy, zy, xw \in E(G^*) $. 
		
		For the sake of contradiction, assume that $ G^* $ is a derived graph, and consider its orientation as an oriented derived graph.  Notice that $H_l $ and $H_r $ are also derived graph. But neither of $ H_r $ and $ H_l $ are chandeliers, so by Theorem~\ref{thm:starcutset-chandelier-degree1}, each of them should have an in-star cutset. Now by Lemma~\ref{lem:unique-in-star-in-K4-subd}, the center of the in-star cutset of $ H_l $ can only be $ y $, so the edge $ xy $ should be oriented from $ x $ to $ y$ in $ G^*$. On the other hand, again by Lemma~\ref{lem:unique-in-star-in-K4-subd}, the center of the in-star cutset of $ H_r $ can only be $ x $, so the edge $ xy $ should be oriented from $ y $ to $ x$ in $ G^*$. This contradiction shows that $ G^* $ is not a Burling graph.
		
		The second part of the theorem follows directly from Lemma~\ref{theorem:the-method}.
	\end{proof}
	
%	\begin{corollary}
%		The graph $ G $ shown in Figure~\ref{fig:twin-K4} is a non-weakly pervasive graph. 
%	\end{corollary}
%	
%	\begin{proof}
%		The proof is immediate from Theorem~\ref{thm:TwinK4-no-BG}. 
%	\end{proof}
%	
	\begin{remark}
		Theorem~\ref{thm:TwinK4-no-BG} provides an example for the fact that the class of Burling graphs is not closed under gluing along one edge. The graph in Figure~\ref{fig:twin-K4} is not a derived graph and is obtained by gluing $ H_r $ and $ H_l $ (in the proof of Theorem~\ref{thm:TwinK4-no-BG}) along one edge, and both of $ H_r $ and $ H_l $ are derived graphs, as shown in Figure~8 of \cite{Part2}.
	\end{remark}

	\subsection*{Second example}

	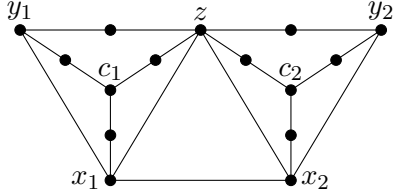
\begin{figure}
		\begin{center}
			\begin{tikzpicture}[scale=.4]
			% middle vertex and edge
			\filldraw[black] (0, 5) circle (5pt) node[anchor=south] {$z $};
			\draw[black] (-3,0) -- (3,0);
			% righ part
			% main:
			\filldraw[black] (3, 3) circle (5pt) node[anchor=south] {$c_2 $};
			\filldraw[black] (6, 5) circle (5pt) node[anchor=south] {$y_2 $};
			\filldraw[black] (3, 0) circle (5pt) node[anchor=west] {$x_2$};
			% rest:
			\filldraw[black] (4.5, 4) circle (5pt);
			\filldraw[black] (1.5, 4) circle (5pt);
			\filldraw[black] (3, 1.5) circle (5pt);
			\filldraw[black] (3, 5) circle (5pt);
			% edges:
			\draw[black] (3,0) -- (6,5) -- (0,5) -- (3,0) -- (3,3) -- (6,5);
			\draw[black] (3,3) -- (0,5);
			% left part
			% main:
			\filldraw[black] (-3, 3) circle (5pt) node[anchor=south] {$c_1 $};
			\filldraw[black] (-6, 5) circle (5pt) node[anchor=south] {$y_1 $};
			\filldraw[black] (-3, 0) circle (5pt) node[anchor=east] {$x_1$};
			% rest:
			\filldraw[black] (-4.5, 4) circle (5pt);
			\filldraw[black] (-1.5, 4) circle (5pt);
			\filldraw[black] (-3, 1.5) circle (5pt);
			\filldraw[black] (-3, 5) circle (5pt);
			% edges:
			\draw[black] (-3,0) -- (-6,5) -- (0,5) -- (-3,0) -- (-3,3) -- (-6,5);
			\draw[black] (-3,3) -- (0,5);
			\end{tikzpicture}
		\end{center}
		%		\centering
		%		\includegraphics[width=6cm]{Pictures/example2.jpg}
		\caption{A graph that is not weakly pervasive. Theorem~\ref{thm:second-twin-K4-non-Burling}.} \label{fig:second-twin-K4}
	\end{figure}
	\begin{theorem} \label{thm:second-twin-K4-non-Burling}
		Let $ G $ be the graph shown in Figure~\ref{fig:second-twin-K4}. No subdivision of $ G $ is a derived graph. In particular, $ G $ is not weakly pervasive. 
	\end{theorem}
	
	\begin{proof}
		Let $ G^* $ be a subdivision of $ G $. If any of the edges $ x_1y_1$, $x_1 z$,  $ x_2y_2$, and $x_2 z$ are subdivided in $ G^* $, then by Lemma~\ref{thm:subdivisions-of-K4}, it contains a non-Burling subdivision of $K_4 $ and thus is not a derived graph. So, we may assume that all those 4 edges are edges of $ G^*$, in which case, if $ x_1 x_2 \in E(G^*)$, then $ G^* $ has a triangle and thus is not a Burling graph. So, we may also assume that $ x_1 x_2 $ is subdivided in $ G^*$.
		
		For the sake of contradiction, consider any orientation on $ G^* $ which makes it an oriented derived graph, and let $ A(G^*) $ be the set of the arcs of $ G^* $. Let $ i \in \{1,2\}$. By Lemma~\ref{lem:K-4-pre-subordinate}, $ y_ix_i, zx_i \in A(G^*)$. Also, by Lemma~\ref{lem:K-4-pre-subordinate}, $ x_i $ is the pivot of exactly one of the following two holes containing the edge~$ zx_i $: the hole containing the paths obtained by subdividing $ x_ic_i $ and $ c_i z$, and the hole containing $ x_iy_i $ and the path obtained by subdividing the $ y_i z$. Denote this hole by $ H_i $. Now consider the hole $ H $ formed by $ zx_1$, $ zx_2 $, and the path obtained by the subdivision of the edge $ x_1x_2$. The two holes $ H $ and $H_i $ form a domino. Therefore, by Lemma~\ref{thm:domino}, for some $ u \in \{x_i, z\} $, $ u $ is a subordinate vertex of one of the holes and the pivot of the other. Notice that $ u $ cannot be $ z $ because $ zx_i \in A(G^*)$. Thus $ u=x_i $, and as $ x_i$ is the pivot of $ H_i $, it is a subordinate vertex of $ H$.
		
		Therefore, in $ H $, the vertex $ z $ is a source and both its neighbors, namely $ x_1 $ and $x_2 $, are subordinate vertices of $ H$. This contradicts Lemma~\ref{thm:holes-4-extrema}. Thus, $ G^* $ is not a derived graph. 
		
		The second part of the theorem follows directly from Lemma~\ref{theorem:the-method}.
	\end{proof}

%	\begin{corollary}
%		The graph $ G $ shown in Figure~\ref{fig:second-twin-K4} is a non-weakly pervasive graph. 
%	\end{corollary}
%
%	\begin{proof}
%		The proof is immediate from Theorem~\ref{thm:second-twin-K4-non-Burling}. 
%	\end{proof}

	\subsection*{Third example}

	\begin{figure}
		\begin{center}
			\begin{tikzpicture}[scale=.35]
			% the K4's
			% righ part
			% main:
			\filldraw[black] (3, 5) circle (5pt) node[anchor=south] {$z_2$};
			\filldraw[black] (6, 3) circle (5pt) node[anchor=south] {$c_2 $};
			\filldraw[black] (9, 5) circle (5pt) node[anchor=south] {$y_2 $};
			\filldraw[black] (6, 0) circle (5pt) node[anchor=west] {$x_2$};
			% rest:
			\filldraw[black] (7.5, 4) circle (5pt);
			\filldraw[black] (4.5, 4) circle (5pt);
			\filldraw[black] (6, 1.5) circle (5pt);
			\filldraw[black] (6, 5) circle (5pt);
			% edges:
			\draw[black] (6,0) -- (9,5) -- (3,5) -- (6,0) -- (6,3) -- (9,5);
			\draw[black] (6,3) -- (3,5);
			% left part
			% main:
			\filldraw[black] (-3, 5) circle (5pt) node[anchor=south] {$z_1$};
			\filldraw[black] (-6, 3) circle (5pt) node[anchor=south] {$c_1 $};
			\filldraw[black] (-9, 5) circle (5pt) node[anchor=south] {$y_1 $};
			\filldraw[black] (-6, 0) circle (5pt) node[anchor=east] {$x_1$};
			% rest:
			\filldraw[black] (-7.5, 4) circle (5pt);
			\filldraw[black] (-4.5, 4) circle (5pt);
			\filldraw[black] (-6, 1.5) circle (5pt);
			\filldraw[black] (-6, 5) circle (5pt);
			% edges:
			\draw[black] (-6,0) -- (-9,5) -- (-3,5) -- (-6,0) -- (-6,3) -- (-9,5);
			\draw[black] (-6,3) -- (-3,5);
			% the theta
			\filldraw[black] (-3, 0) circle (5pt) node[anchor=south] {$w_1$};
			\filldraw[black] (-3, -4) circle (5pt);
			\filldraw[black] (3, 0) circle (5pt) node[anchor=south] {$w_2$};
			\filldraw[black] (3, -4) circle (5pt);
			\draw[black] (-3,5) -- (3,5);
			\draw[black] (-6,0) -- (6,0);
			\draw[black] (-6,0) -- (-3,-4) -- (3,-4) -- (6,0);
			\end{tikzpicture}
		\end{center}
		\caption{A graph that is not weakly pervasive. Theorem~\ref{thm:theta-K4-non-Burling}.} \label{fig:theta-K4}
	\end{figure}
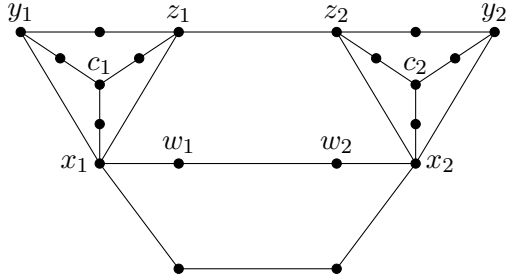

	\begin{theorem} \label{thm:theta-K4-non-Burling}
		Let $ G $ be the graph shown in Figure~\ref{fig:theta-K4}. No subdivision of $ G $ is a derived graph. In particular, $ G $ is not weakly pervasive.
	\end{theorem}
	
	\begin{proof}
		Let $ G^* $ be a subdivision of $ G $. If any of the edges $ x_1y_1$, $x_1 z_1$,  $ x_2y_2$, and $x_2 z_2$ are subdivided in $ G^* $, then by Lemma~\ref{thm:subdivisions-of-K4}, it contains a non-Burling subdivision of $K_4 $ and thus is not a derived graph. So, we may assume that all those 4 edges are edges of $ G^*$.
		
		For the sake of contradiction, consider any orientation on $ G^* $ which makes it an oriented derived graph, and let $ A(G^*) $ be the set of the arcs of $ G^* $. Let $ i \in \{1,2\}$. By Lemma~\ref{lem:K-4-pre-subordinate}, $ y_ix_i, z_ix_i \in A(G^*)$. Again By Lemma~\ref{lem:K-4-pre-subordinate}, $ x_i $ is the pivot of exactly one of the following two holes containing the edge~$ z_ix_i $: the hole containing the paths obtained by subdividing $ x_ic_i $ and $ c_i z_i$, and the hole containing $ x_iy_i $ and the path obtained by subdividing the $ y_i z_i$. Denote this hole by $ H_i $. Now consider the hole $ H $ passing through $ z_1x_1$, $ z_2x_2 $, the path obtained by the subdivision of the edge $ z_1z_2$, and the path $ x_1 w_1 \dots w_2 x_2$. The two holes $ H $ and $H_i $ form a domino. Therefore, by Lemma~\ref{thm:domino}, for some $ u \in \{x_i, z_i\} $, $ u $ is the subordinate vertex of one of the holes and the pivot of the other. Notice that $ u $ cannot be $ z $ because $ z_ix_i \in A(G^*)$. Thus $ u=x_i $, and as $ x_i$ is the pivot of $ H_i $, it is a subordinate vertex of $ H$.
		But by Lemma~\ref{lem:theta-pivot}, either $ x_1 $ or $x_2 $ should be pivot of every hole in the long theta in $ G^* $. This contradiction completes the proof. 
		
		The second part of the theorem follows directly from Lemma~\ref{theorem:the-method}.
	\end{proof}
%	\begin{corollary}
%		The graph $ G $ shown in Figure~\ref{fig:theta-K4} is a non-weakly pervasive graph. 
%	\end{corollary}
%
%	\begin{proof}
%		The proof is immediate from Theorem~\ref{thm:theta-K4-non-Burling}. 
%	\end{proof}

%	\bibliographystyle{plain}
%	\bibliography{bibliography}

\end{document}